\documentclass[11pt]{amsart}

\usepackage[letterpaper,left=1.5in,right=1.5in]{geometry}

\usepackage[english]{babel}
\usepackage{amsmath,amssymb,amsfonts,amsthm}
\usepackage{graphicx}

\usepackage[colorlinks,linkcolor=blue,citecolor=blue,urlcolor=blue]{hyperref}

\allowdisplaybreaks

\theoremstyle{plain}
\newtheorem{theorem}{Theorem}[section]
\newtheorem{lemma}[theorem]{Lemma}
\newtheorem*{lemma*}{Lemma}
\newtheorem{proposition}[theorem]{Proposition}
\newtheorem{corollary}[theorem]{Corollary}

\theoremstyle{definition}
\newtheorem{definition}[theorem]{Definition}

\theoremstyle{remark}
\newtheorem{remark}[theorem]{Remark}
\newtheorem{example}[theorem]{Example}

\numberwithin{equation}{section}

\usepackage[style=alphabetic]{biblatex}
\AtEveryBibitem{\clearfield{url}\clearfield{doi}\clearfield{issn}}
\AtEveryBibitem{\clearfield{urldate}\clearfield{urlyear}\clearfield{urlmonth}}

\usepackage[shortlabels]{enumitem}

\usepackage{tikz}
\usepackage{tikz-cd}

\usetikzlibrary{decorations.pathreplacing,calligraphy,calc,positioning}

\makeatletter
\tikzset{
    block filldraw/.style={
        draw, fill=white, line width=0.67pt},
    block rect/.style={
        block filldraw, rectangle},
    block/.style={
        block rect, minimum height=0.8cm, minimum width=6em},
    from/.style args={#1 to #2}{
        above right={0cm of #1},
        /utils/exec=\pgfpointdiff
            {\tikz@scan@one@point\pgfutil@firstofone(#1)\relax}
            {\tikz@scan@one@point\pgfutil@firstofone(#2)\relax},
        minimum width/.expanded=\the\pgf@x,
        minimum height/.expanded=\the\pgf@y}}
\makeatother

\usepackage[outline]{contour}
\contourlength{3pt}

\makeatletter
\newcommand\xleftrightarrow[2][]{%
  \ext@arrow 9999{\longleftrightarrowfill@}{#1}{#2}}
\newcommand\longleftrightarrowfill@{%
  \arrowfill@\leftarrow\relbar\rightarrow}
\makeatother

\mathcode`l="8000
\begingroup
\makeatletter
\lccode`\~=`\l
\DeclareMathSymbol{\lsb@l}{\mathalpha}{letters}{`l}
\lowercase{\gdef~{\ifnum\the\mathgroup=\m@ne \ell \else \lsb@l \fi}}%
\makeatother%
\endgroup

\newcommand{\bbZ}{\mathbb{Z}}
\newcommand{\bbQ}{\mathbb{Q}}
\newcommand{\bbR}{\mathbb{R}}
\newcommand{\calG}{\mathcal{G}}
\newcommand{\calA}{\mathcal{A}}
\newcommand{\inv}{^{-1}}
\newcommand{\half}{^{1/2}}
\newcommand{\nhalf}{^{-1/2}}

\DeclareMathOperator{\sgn}{sgn}
\DeclareMathOperator{\type}{type}
\DeclareMathOperator{\isgn}{isgn}
\DeclareMathOperator{\Match}{Match}

\newcommand{\allbf}[1]{\textbf{\boldmath{#1}}}

\usepackage{amssymb}

\usepackage{wasysym}
\newcommand{\unknot}{\fullmoon}

\renewcommand{\tilde}{\widetilde}

\newcommand{\tbleft}{%
\node at (0.26,0) {\includegraphics{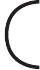}};
\node at (0.26,2) {\includegraphics{graphics/2b_capl}};}
\newcommand{\tbbotblank}[1]{%
\node at (#1,0) {\includegraphics{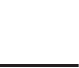}};
\node at (#1,2) {\rotatebox{180}{\includegraphics{graphics/2b0_bot}}};}
\newcommand{\tbbotpos}[1]{%
\tbbotblank{#1} \node at (#1,1) {\includegraphics{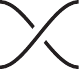}};}
\newcommand{\tbbotneg}[1]{%
\tbbotblank{#1} \node at (#1,1) {\includegraphics{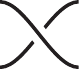}};}
\newcommand{\tbbotzero}[1]{%
\tbbotblank{#1} \node at (#1,1) {\includegraphics{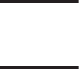}};}

\newcommand{\tbbotbox}[2]{%
\tbbotzero{#1}{}
\tbbotzero{#1+1}{}
\node[block,from={#1-0.35,0.3 to #1+0.58+0.35,1.2}] {$#2$};}
\newcommand{\tbtopblank}[1]{%
\node at (#1,0) {\includegraphics{graphics/2b0_bot}};
\node at (#1,1) {\includegraphics{graphics/2b0_bot}};}

\newcommand{\tbtopneg}[1]{%
\tbtopblank{#1} \node at (#1,2) {\includegraphics{graphics/2b-}};}
\newcommand{\tbtopzero}[1]{%
\tbtopblank{#1} \node at (#1,2) {\includegraphics{graphics/2b0}};}

\newcommand{\tbtopbox}[2]{%
\tbtopzero{#1}{}
\tbtopzero{#1+1}{}
\node[block,from={#1-0.35,1.3 to #1+0.58+0.35,2.2}] {$#2$};}
\newcommand{\tbdots}[1]{%
\node at (#1+0.5,1+0.01) {\includegraphics{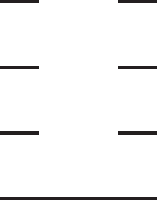}};
\node at (#1+0.5,0.5) {$\cdots$};
\node at (#1+0.5,1.5) {$\cdots$};
\node at (#1+0.5,2.5) {$\cdots$};}
\newcommand{\tboddend}[1]{%
\node at (#1-0.25,0) {\includegraphics{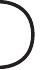}};
\node at (#1-0.25,2) {\includegraphics{graphics/2b_capr}};}
\newcommand{\tbevenend}[1]{%
\node at (#1,1) {\includegraphics{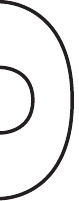}};}

\newcommand{\includegraphicsinline}[1]{\begin{gathered}\includegraphics{#1}\end{gathered}}

\title{Cluster Algebras and the HOMFLY Polynomial}

\author[M. Yacavone]{Matthew Yacavone} \address{Haverford College, Haverford, PA 19041} \email{matthew@yacavone.net}

\thanks{Code exploring the results of this paper can be found at: \url{https://github.com/m-yac/F-polys}}

\begin{document}

\begin{abstract}
Recently, it has been shown that the Jones polynomial, in \cite{lee_cluster_2019}, and the Alexander polynomial, in \cite{nagai_cluster_2018}, of rational knots can be obtained by specializing $F$-polynomials of cluster variables. At the core of both results are continued fractions, which parameterize rational knots and are used to obtain cluster variables, by way of snake graphs in the case of \cite{lee_cluster_2019}, or ancestral triangles in the case of \cite{nagai_cluster_2018}. In this paper, we use path posets, another structure parameterized by continued fractions, to directly generalize \cite{lee_cluster_2019}'s construction to a specialization yielding the HOMFLY polynomial, which generalizes both the Jones and Alexander polynomials.
\end{abstract}

\maketitle

\section{Introduction}

The core question of knot theory is how to determine, in general, whether a given knot is equivalent to another. 
Invariants, functions which respect equivalence of knots, are the main way we try to address this question, and polynomial invariants are one of the most important classes of invariants. In the 1920s, Alexander \cite{alexander_topological_1928} introduced the first of these, the Alexander polynomial $\Delta(K)$. This remained the only invariant of its kind until 1985, when the introduction of Jones' \cite{jones_polynomial_1985} polynomial $V(K)$ spurred the creation of numerous others, including the HOMFLY polynomial $P(K)$ \cite{freyd_new_1985}, which generalizes both previously mentioned.

In this paper, we discuss how to compute these polynomial invariants for rational knots, a family of knots indexed by rational numbers, by way of a surprising connection to cluster algebras. Cluster algebras were first introduced in 2002 by Fomin and Zelevinsky \cite{fomin_cluster_2002} for use in Lie theory, and have since been found to connect to many different parts of mathematics, including representation theory, combinatorics, algebraic geometry, and mathematical physics. In 2017, Lee and Schiffler \cite{lee_cluster_2019} presented a way to compute the Jones polynomial of rational knots by way of computing the $F$-polynomial of a cluster variable in an associated cluster algebra. In particular, they gave a specialization for the variables of the $F$-polynomial which yields the Jones polynomial of rational knots. The following year, Nagai and Terashima \cite{nagai_cluster_2018} gave a specialization of the $F$-polynomial which yields the Alexander polynomial of rational knots. In this paper, we present a specialization which yields the HOMFLY polynomial.

At the center of the connection between these polynomial invariants and cluster algebras are continued fractions, representations of rational numbers as lists of integers. It is natural to reason about rational knots using continued fractions, and in \cite{lee_cluster_2019}, \cite{nagai_cluster_2018}, and this paper, the strategy is to use another family of objects similarly indexed by continued fractions to compute the $F$-polynomial. In \cite{lee_cluster_2019}, these objects are snake graphs, first introduced in the context of cluster algebras in \cite{musiker_cluster_2010} as a computational tool, and later explored in \cite{musiker_positivity_2011}, \cite{musiker_bases_2013}, \cite{canakci_snake_2013}, and \cite{canakci_cluster_2018}. In \cite{nagai_cluster_2018}, these objects are ancestral triangles, whose relation to cluster algebras is explored in their paper. 
This paper provides a third perspective on $F$-polynomials, that of path posets. These posets were briefly mentioned, though left unnamed, in \cite{musiker_bases_2013}, and later explored further in \cite{bailey_cluster_2019}, where they go by the name `piecewise-linear.' 

\subsection{Statement of Main Result}

In this paper we directly generalize \cite{lee_cluster_2019}'s work to the case of the HOMFLY polynomial. Before we can state this result, we need to introduce some notation. For any knot or link $L$, the HOMFLY polynomial $P(L) \in \bbZ(l,q\half)$ is a rational function in two variables $l$ and $q\half$ with integer coefficients.\footnote{\label{fnote:HOMFLY-kinds}There are many well-known equivalent formulations of the HOMFLY polynomial. We use a slight modification of the one presented in \cite{jones_hecke_1987}, where we make the substitution $\lambda = l\inv q^{-1/2}$.} We write $P[b_1,\ldots,b_n]$ to refer to the HOMFLY polynomial of the rational knot corresponding to the continued fraction $[b_1,\ldots,b_n]$. The $F$-polynomial is a Laurent polynomial in $n$ variables $y_1,\ldots,y_n$ with integer coefficients, and we write $F[b_1,\ldots,b_n]$ to refer to the $F$-polynomial of the path poset corresponding to $[b_1,\ldots,b_n]$. In Section~\ref{ssec:homfly_formulae} we inductively define $m[b_1,\ldots,b_n] \in \bbZ(l,q\half)$, and only remark here that it is always equal to $c_0\,l^{e_1}q^{e_2}((1-l^2q)/(1-q\inv))^{e_3}$ for some $e_1,e_2 \in \bbZ$, $e_3 \in \{-1,0,1\}$, and $c_0 \in \{-1,1\}$.

\begin{theorem}\label{thm:B}
For any even continued fraction $[b_1,\ldots,b_n]$, \[ P[b_1,\ldots,b_n] = 
m[b_1,\ldots,b_n]\,
\tilde{F}[b_1,\ldots,b_n] \]
where $\tilde{F}[b_1,\ldots,b_n]$ is the $F$-polynomial $F[b_1,\ldots,b_n]$ subject to the specialization
\begin{align*}
y_1 &= l^2\left(\frac{1 - l^2q}{1 - q\inv}\right)^{\hspace{-2pt}-1} \text{ if $b_1 > 0$} \enspace\text{ or }\enspace q^{-2}\left(\frac{1 - l^2q}{1 - q\inv}\right) \text{ if $b_1 < 0$,}
\end{align*}
\begin{align*}
y_{2i} &= -l^2q, & y_{2i+1} &= -q\inv & &\text{for all } i \geq 1.
\end{align*}
\end{theorem}

If we substitute $t\inv$ for $l$ and $t$ for $q$, we get exactly the specialization presented in \cite{lee_cluster_2019}. Indeed, the Jones polynomial $V(L)$ is equal to $P(L)|_{l = t\inv, q = t}$ \cite{jones_hecke_1987}. Furthermore, if we apply this substitution to $m[b_1,\ldots,b_n]$ 
we get exactly the leading term of the Jones polynomial as computed in \cite{lee_cluster_2019}; see Section~\ref{ssec:homfly_formulae}. For the reader that is familiar with the work of \cite{lee_cluster_2019} or \cite{nagai_cluster_2018}, it is important to note that our result is dependent on the specific way we label a path poset (equivalently, snake graph) and demand a labeled path poset (or labeled snake graph) be realized as a cluster variable; see Definitions~\ref{def:path_poset_cf} and \ref{def:realize}. This difference is also explicitly discussed before Remark~\ref{rmk:Sn-labels} and Lemma~\ref{lem:Sn-Fpoly}.

Since the HOMFLY polynomial also generalizes the Alexander polynomial $\Delta(L) \in \bbZ[t^{\pm1/2}]$ by substituting $1$ for $l$ and $t$ for $q$ \cite{jones_hecke_1987}, as a corollary we get a statement analogous to that in \cite{nagai_cluster_2018} by applying this substitution. We write $\Delta[b_1,\ldots,b_n]$ to refer to the Alexander polynomial of the rational knot corresponding to the continued fraction $[b_1,\ldots,b_n]$.

\begin{corollary}\label{cor:C}
For any even continued fraction $[b_1,\ldots,b_n]$,
\[ \Delta[b_1,\ldots,b_n] = \sgn(c_0)\,t^{e_0} \tilde{F}[b_1,\ldots,b_n] \]
where $c_0t^{e_0}$ is the leading term of $\Delta[b_1,\ldots,b_n]$ and $\tilde{F}[b_1,\ldots,b_n]$ is the $F$-polynomial $F[b_1,\ldots,b_n]$ subject to the specialization $y_i = -t^{(-1)^i}$.
\end{corollary}

We remark that while the specialization presented in \cite{nagai_cluster_2018} takes five cases to define and is dependent on the particular continued fraction expansion, ours takes only one and works in all cases. 

\subsection{Organization}

In Section~\ref{sec:bkg} we review relevant background on continued fractions, rational knots, and the HOMFLY polynomial. In Section~\ref{sec:pps} we introduce path posets, and in Section~\ref{sec:cas} show how they relate to cluster algebras and the $F$-polynomial. Finally, in Section~\ref{sec:homfly_spec}, we prove Theorem~\ref{thm:B} and Corollary~\ref{cor:C}.

\section{Background}\label{sec:bkg}

In this section we introduce in more detail continued fractions, rational knots, and the HOMFLY, Jones, and Alexander polynomials. This section then concludes with a brief discussion of the HOMFLY polynomial of rational knots.

\subsection{Continued Fractions}

Define the set
\[ \bbQ^\infty := \left\{ \frac{p}{q} : p \in \bbZ_{\geq 0}, q \in \bbZ \text{\;\;s.t.\;} \gcd(p,q) = 1
\right\}
= \bbQ \cup \left\{1/0\right\}
\]
where $\infty := 1/0$ satisfies $x + \infty = \infty$ for all $x \in \bbQ^\infty$, 
as in \cite{duzhin_formula_2015} or \cite{kauffman_classification_2003}.

For any $p/q \in \bbQ^\infty$, a \textbf{continued fraction expansion} for $p/q$, or just a \textbf{continued fraction}, is a nonempty list $c_1,\ldots,c_n$ of integers such that
\[ [c_1,\ldots,c_n] := c_1 + \frac{1}{c_2 + \frac{1}{\cdots + \frac{1}{c_n}}} = \frac{p}{q}. \]

Following the notation used in \cite{lee_cluster_2019}, we extend the definition of a continued fraction to include the degenerate $n = 0$ and $n = -1$ cases as follows:
\begin{equation*}
[\,] := \infty \quad(n = 0) \quad\quad\quad\quad [c_1,\ldots,c_{-1}] := 0\quad(n = -1).
\end{equation*}

A continued fraction $[a_1,\ldots,a_n]$ is \textbf{positive} if each $a_i$ is strictly positive, and a continued fraction $[b_1,\ldots,b_n]$ is \textbf{even} if each $b_i$ is even and nonzero. 
The following proposition characterizes when such expansions exist.

\begin{proposition}[\cite{hardy_introduction_1979},\cite{lee_cluster_2019}]\label{prop:pos-even}
Suppose $n \geq 0$ and $p/q \in \bbQ^\infty$.\begin{enumerate}[(a)]
\item\label{item:pos} There is a positive continued fraction $[a_1,\ldots,a_n] = p/q$ if and only if \mbox{$p/q \geq 1$}. This expansion is unique up to the following equality.
\begin{equation}\label{eqn:last_1}
[c_1,\ldots,c_n,1] = c_1 + \frac{1}{\cdots + \frac{1}{c_{n} + \frac{1}{1}}} = c_1 + \frac{1}{\cdots + \frac{1}{c_{n} + 1}} = [c_1,\ldots,c_n+1]
\end{equation}
\item\label{item:even} There is an even continued fraction $[b_1,\ldots,b_n] = p/q$ if and only if \mbox{$|p/q| \geq 1$} and either $p$ or $q$ is even. This expansion is always unique.
\end{enumerate}
\end{proposition}

The map $p/q \mapsto p/(q-\sgn(q)\,p)$, found in both \cite{duzhin_formula_2015} and \cite{murasugi_knot_2007}, distinguishes both positive and even continued fraction expansions in the following way.

\begin{remark}\label{rmk:invol}
As a consequence of Proposition~\ref{prop:pos-even}, for any $p/q \in \bbQ$ s.t.\;$|p/q| > 1$, \[ \begin{gathered}
p/q \text{ has a positive}\\\text{C.F. expansion}
\end{gathered}
\enspace\Longleftrightarrow\enspace
\begin{gathered}
p/(q - \sgn(q)\,p) \text{ does not have a}\\\text{positive C.F. expansion.}
\end{gathered} \]
The same holds replacing ``positive'' with ``even'' in the above.
\end{remark}

There are a few quantities defined on continued fractions which we will use often in this paper. These can be found in \cite{lee_cluster_2019}, \cite{canakci_cluster_2018}, and \cite{rabideau_f-polynomial_2018} (among others) defined only for positive or even continued fractions. We extend them here to all continued fractions.
\begin{definition}\label{def:cf_quans}
For any continued fraction $[c_1,\ldots,c_n]$, we define
\begin{align*}
l_0 &:= 0 & l_1 &:= |c_1| & l_2 &:= |c_1| + |c_2| & &\cdots & l_n &:= |c_1| + |c_2| + \cdots + |c_n|,
\end{align*}
the \textbf{type sequence}
\begin{align*}
t_0 &:= -1 & t_1 &:= \sgn(c_1) & t_2 &:= -\sgn(c_2) & &\cdots & t_n &:= (-1)^{n-1} \sgn(c_n),
\end{align*}
written $\type[c_1,\ldots,c_n] := (t_1,\ldots,t_n)$, and the \textbf{sign sequence}
\begin{align*}
\sgn[c_1,\ldots,c_n] := (\underbrace{t_1,\ldots,t_1}_{|c_1|},\underbrace{t_2,\ldots,t_2}_{|c_2|},\ldots,\underbrace{t_n,\ldots,t_n}_{|c_n|}).
\end{align*}
\end{definition}

Given a continued fraction with $\ell_n \geq 2$, removing the first and last terms from $\sgn[c_1,\ldots,c_n] = (s_1,\ldots,s_{\ell_n})$ gives what we will call the \textbf{inner sign sequence}
\[ \isgn[c_1,\ldots,c_n] := (s_2,\ldots,s_{l_n-1}) \in \{1,-1\}^{l_n-2}. \]
The following proposition, which is based on results in \cite{canakci_cluster_2018}, uses the inner sign sequence to establish a key way of understanding positive continued fractions.

\begin{proposition}\label{prop:seqs-pos-cf}
The inner sign sequence defines a bijection between positive continued fractions with $\ell_n \geq 2$ and sequences in $\{1,-1\}^{\ell_n-2}$. 
As a result, there is a bijection:
\[ \left\{ p/q \in \bbQ : p/q > 1 \right\} \xleftrightarrow{\hphantom{aa}} \bigcup_{k \geq 0}\; \{1,-1\}^k. \]
\end{proposition}

\begin{proof}[Sketch of Proof]
The first bijection follows from arguments in \cite{canakci_cluster_2018} and the observations that first term of the sign sequence is always $1$ and the last term is the only term changed by \eqref{eqn:last_1}. The second bijection follows from the fact that only $1$ and $\infty$ have positive continued fraction expansions with $l_n < 2$.
\end{proof}

\subsection{Rational Knots and Links}\label{ssec:knots}

Briefly, a \textbf{link} is a smooth embedding of $n \geq 1$ disjoint copies of $S^1$ into $\bbR^3$, and a \textbf{knot} is a link with one component. Knots and links are considered up to \textbf{isotopy}, where links $L,L'$ are \textbf{isotopic} if there is an orientation-preserving diffeomorphism of $\bbR^3$ which takes the image of $L$ to $L'$. We often study a link by looking at \textbf{diagrams}, projections of the link onto $\bbR^2$ with over/under crossings marked. For more detail or a standard reference on the subject, see \cite{murasugi_knot_2007}, for example.

To every $p/q \in \bbQ^\infty$ we can associate in an algebraic way a link $C(p/q)$; see \cite{kauffman_classification_2003} or \cite{murasugi_knot_2007}. Of use to us is that these rational links $C(p/q)$ have the following presentation given a continued fraction expansion.

\begin{theorem}[\cite{murasugi_knot_2007} or \cite{kauffman_classification_2003}]\label{thm:rational}
For any continued fraction $[c_1,\ldots,c_n]$ with $n \geq 1$, the rational link $C\big([c_1,\ldots,c_n]\big)$ has the diagram
\begin{align*}
&\begin{gathered}\begin{tikzpicture}[yscale=0.6667,xscale=0.8]
\tbleft
\tbbotbox{1}{t_1|c_1|}
\tbtopbox{3}{t_2|c_2|}
\tbdots{5}
\tbbotbox{7}{t_n|c_n|}
\tboddend{9}
\end{tikzpicture}\end{gathered} &&\text{if $n$ odd }\; \\
&\begin{gathered}\begin{tikzpicture}[yscale=0.6667,xscale=0.8]
\tbleft
\tbbotbox{1}{t_1|c_1|}
\tbtopbox{3}{t_2|c_2|}
\tbdots{5}
\tbtopbox{7}{t_n|c_n|}
\tbevenend{9}
\end{tikzpicture}\end{gathered} &&\text{if $n$ even }\;
\end{align*}
where for any $k > 0$, we define the tangles
\begin{gather*}
\begin{gathered}\begin{tikzpicture}[yscale=0.6667,xscale=0.8]
\node at (0,1) {\includegraphics{graphics/2b0}};
\node at (1,1) {\includegraphics{graphics/2b0}};
\node[block,from={-0.35,0.3 to 0.58+0.35,1.2}] {$k$};
\node at (2,1) {$=$};
\node at (3,1) {\includegraphics{graphics/2b+}};
\node at (4,1) {$\cdots$};
\node at (5,1) {\includegraphics{graphics/2b+}};
\node at (4,0-0.75) {\small $k$ crossings};
\draw[decorate, decoration={calligraphic brace,amplitude=5pt}, line width=1pt]
    ( $ (6-0.5,0) $ ) -- ( $ (2+0.5,0) $ );
\end{tikzpicture}\end{gathered}
\quad\quad\quad\quad
\begin{gathered}\begin{tikzpicture}[yscale=0.6667,xscale=0.8]
\node at (0,1) {\includegraphics{graphics/2b0}};
\node at (1,1) {\includegraphics{graphics/2b0}};
\node[block,from={-0.35,0.3 to 0.58+0.35,1.2}] {$-k$};
\node at (2,1) {$=$};
\node at (3,1) {\includegraphics{graphics/2b-}};
\node at (4,1) {$\cdots$};
\node at (5,1) {\includegraphics{graphics/2b-}};
\node at (4,0-0.75) {\small $k$ crossings};
\draw[decorate, decoration={calligraphic brace,amplitude=5pt}, line width=1pt]
    ( $ (6-0.5,0) $ ) -- ( $ (2+0.5,0) $ );
\end{tikzpicture}\end{gathered} \\
\begin{gathered}\begin{tikzpicture}[yscale=0.6667,xscale=0.8]
\node at (0,1) {\includegraphics{graphics/2b0}};
\node at (1,1) {\includegraphics{graphics/2b0}};
\node[block,from={-0.35,0.3 to 0.58+0.35,1.2}] {$0$};
\node at (2,1) {$=$};
\node at (3,1) {\includegraphics{graphics/2b0}};
\end{tikzpicture}\end{gathered}
\end{gather*}
so that the diagram as defined has $\ell_n$ crossings.
\end{theorem}

Some examples of rational links are shown in Figure~\ref{fig:link_exs}.

\begin{figure}[ht]
$\begin{gathered}\scalebox{0.95}{\begin{tikzpicture}[yscale=0.6667,xscale=0.8]
\tbleft{}
\tbbotpos{1}\tbbotpos{2}
\node at (1.5,0) {$2$};
\tbtopneg{3}\tbtopneg{4}\tbtopneg{5}
\node at (4,1) {$-3$};
\tbbotneg{6}\tbbotneg{7}\tbbotneg{8}\tbbotneg{9}
\node at (7.5,0) {$-4$};
\tbtopneg{10}\tbtopneg{11}
\node at (10.5,1) {$-2$};
\tbbotpos{12}\tbbotpos{13}\tbbotpos{14}
\node at (13,0) {$3$};
\tbtopneg{15}
\node at (15,1) {$-1$};
\tbevenend{16}
\end{tikzpicture}}\end{gathered}$
$\begin{gathered}\scalebox{0.95}{\begin{tikzpicture}[yscale=0.6667,xscale=0.8]
\tbleft{}
\tbbotpos{1}\tbbotpos{2}
\node at (1.5,0) {$2$};
\tbtopneg{3}\tbtopneg{4}
\node at (3.5,1) {$-2$};
\tbbotpos{5}
\node at (5,0) {$1$};
\tbtopneg{6}\tbtopneg{7}
\node at (6.5,1) {$-2$};
\tbbotpos{8}
\node at (8,0) {$1$};
\tbtopneg{9}
\node at (9,1) {$-1$};
\tbbotpos{10}\tbbotpos{11}\tbbotpos{12}\tbbotpos{13}
\node at (11.5,0) {$4$};
\tboddend{14}
\end{tikzpicture}}\end{gathered}$
\caption{The isotopic rational links $C\big([2,3,-4,2,3,1]\big)$ and $C\big([2,2,1,2,1,1,4]\big)$, with their constituent tangles labeled.}
\label{fig:link_exs}
\end{figure}
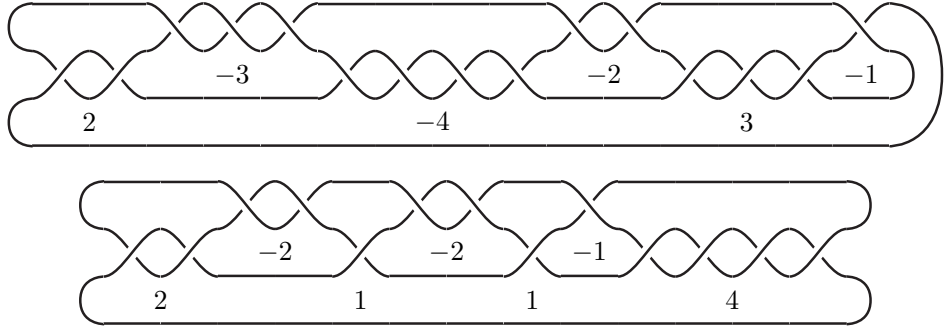

\begin{theorem}[{\cite{kauffman_classification_2003} or \cite{murasugi_knot_2007}}]\label{thm:rational-iso}
The rational links $C(p/q)$ and $C(p'/q')$ are isotopic if and only if 
$p' = p$ and either $q' \equiv q \,\bmod\, p$ or $q' \equiv q\inv \,\bmod\, p$.
\end{theorem}

The following consequence of Theorem~\ref{thm:rational-iso} and Remark~\ref{rmk:invol} 
shows that, up to isotopy, both positive and even continued fractions realize all rational links.

\begin{corollary}\label{cor:rational-pos-even}
For every $p/q \in \bbQ^\infty$, there exists a positive continued fraction $[a_1,\ldots,a_n]$ and an even continued fraction $[b_1,\ldots,b_n]$ such that $C(p/q)$ is isotopic to $C\big([a_1,\ldots,a_n]\big)$ and $C\big([b_1,\ldots,b_n]\big)$.
\end{corollary}

This paper will be concerned with \textbf{oriented} links (see \cite{murasugi_knot_2007}), where to a crossing in an oriented link diagram we associate a sign in the following way.
\begin{align*}
\begin{gathered}
\includegraphicsinline{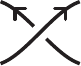} \\ 1 \text{ (`positive')}
\end{gathered} &&
\begin{gathered}
\includegraphicsinline{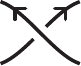} \\ -1 \text{ (`negative')}
\end{gathered}
\end{align*}

The following result, which can also be found in \cite{duzhin_formula_2015}, establishes that there is a well-behaved canonical way to orient a rational link of an even continued fraction. In general, orientations of rational links need not be this nice; see \cite{nagai_cluster_2018}, for example.
\begin{proposition}[\cite{lee_cluster_2019}]\label{prop:rational-ori}
For any even continued fraction $[b_1,\ldots,b_n]$, there exists an orientation of $C\big([b_1,\ldots,b_n]\big)$ such that the sequence of signs of the crossings of $C\big([b_1,\ldots,b_n]\big)$, read left to right, is exactly the sign sequence $\sgn[b_1,\ldots,b_n]$. In particular, the $i^{\text{th}}$ tangle has the following orientation, irrespective of over/under crossings.
\begin{align*}
\begin{gathered}
\begin{gathered}\begin{tikzpicture}[yscale=0.6667,xscale=0.8]
\node at (0-0.125,1) {\includegraphics{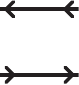}};
\node at (1+0.08,1) {\includegraphics{graphics/2b0_ori1}};
\node[block,from={-0.35,0.3 to 0.58+0.35,1.2}] {$t_i|b_i|$};
\node at (2+0.08,1) {$=$};
\node at (3+0.08,1) {\includegraphics{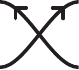}};
\node at (4+0.08,1) {$\cdots$};
\node at (5+0.08,1) {\rotatebox{180}{\includegraphics{graphics/2bnone}}};
\end{tikzpicture}\end{gathered} \\ \text{if $i$ odd}
\end{gathered} &&
\begin{gathered}
\begin{gathered}\begin{tikzpicture}[yscale=0.6667,xscale=0.8]
\node at (0-0.125,1) {\includegraphics{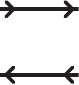}};
\node at (1+0.08,1) {\includegraphics{graphics/2b0_ori2}};
\node[block,from={-0.35,0.3 to 0.58+0.35,1.2}] {$t_i|b_i|$};
\node at (2+0.08,1) {$=$};
\node at (3+0.08,1) {\rotatebox{180}{\includegraphics{graphics/2bnone}}};
\node at (4+0.08,1) {$\cdots$};
\node at (5+0.08,1) {\includegraphics{graphics/2bnone}};
\end{tikzpicture}\end{gathered} \\ \text{if $i$ even}
\end{gathered}
\end{align*}
\end{proposition}

See Figure~\ref{fig:rational-ori-ex} for an example.

\begin{figure}[ht]
\centering
\includegraphics{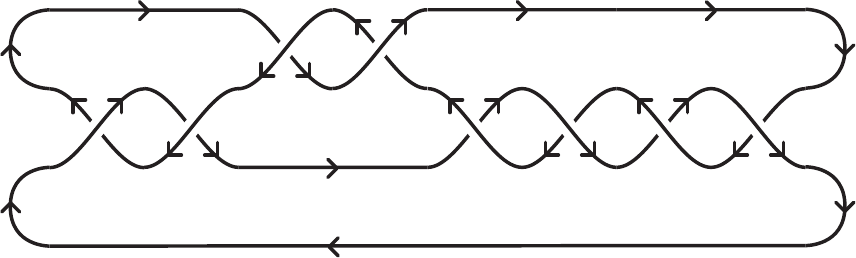}
\caption{The canonical orientation for $C\big([2,-2,-4]\big)$, where $\sgn[2,-2,-4] = (1,1,1,1,-1,-1,-1,-1)$.}
\label{fig:rational-ori-ex}
\end{figure}

\subsection{The HOMFLY Polynomial}\label{ssec:homfly_def}

The following theorem of Vaughan Jones will serve as our definition of the HOMFLY polynomial, though there are many other equivalent formulations. It uses the notion of a \textbf{skein relation}; see \cite{murasugi_knot_2007} for a general reference.

\begin{theorem}[\cite{jones_hecke_1987}\footnote{Relative to \cite{jones_hecke_1987}, we make the substitution $\lambda = l\inv q^{-1/2}$.}]\label{thm:homfly-def}
There exists an isotopy invariant $P$, called the \textbf{HOMFLY polynomial}, such that $P(L) \in \bbZ(l,q^{1/2})$ for any oriented link $L$, and which satisfies the following skein relation:
\[ l\,P\left(\includegraphicsinline{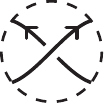}\right) - l\inv P\left(\includegraphicsinline{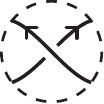}\right) = (q\half - q\nhalf)\, P\left(\includegraphicsinline{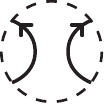}\right) \]
as well as the equality $P(\unknot) = 1$ on the unknot.
\end{theorem}

\begin{corollary}[\cite{jones_hecke_1987}]\label{cor:homfly-unknot}
If $L$ is isotopic to two disjoint copies of the unknot, then
\[ P(L) = \frac{l-l\inv}{q\half - q\nhalf}. \]
\end{corollary}

The HOMFLY polynomial generalizes the Jones polynomial, as defined in \cite{lee_cluster_2019}.

\begin{theorem}[\cite{jones_hecke_1987}]\label{thm:homfly-to-jones}
For any oriented link $L$, the Jones polynomial satisfies \[ V(L) = P(L)|_{l = t\inv, q = t}. \]
\end{theorem}

Furthermore, via a different substitution, we obtain another well-known polynomial which is distinct from the Jones. This polynomial was shown to have an $F$-polynomial specialization in \cite{nagai_cluster_2018}.

\begin{theorem}[\cite{jones_hecke_1987}]\label{thm:homfly-to-alex}
For any oriented link $L$, the {Alexander polynomial} satisfies \[ \Delta(L) = P(L)|_{l = 1, q = t}. \]
\end{theorem}

\subsection{The HOMFLY Polynomial of Rational Knots}\label{ssec:homfly_formulae}

For any oriented rational link $C\big([c_1,\ldots,c_n]\big)$, we define $P[c_1,\ldots,c_n] := P(C\big([c_1,\ldots,c_n]\big))$. Recall from Proposition~\ref{prop:rational-ori} that it will be most natural to work with oriented rational links using even continued fractions, and from Corollary~\ref{cor:rational-pos-even}, that restricting to only even continued fractions loses us no rational links.

We recall formulae for the HOMFLY polynomial $P[b_1,\ldots,b_n]$ of an even continued fraction from \cite{duzhin_formula_2015}. 
First, as in \cite{lee_cluster_2019}, we fix the following notation for $n \geq 1$.
\[ P_0 = P[b_1,\ldots,b_n] \quad\quad P_1 = P[b_1,\ldots,b_{n-1}] \quad\quad P_2 = P[b_1,\ldots,b_{n-2}] \]

\begin{lemma}[{\cite[(3)]{duzhin_formula_2015}\footnote{\label{note:duzhin}Relative to \cite{duzhin_formula_2015}, we make the substitution $a = \ell$, $z = q\half - q\nhalf$ and flip all crossings in $C\big([b_1,\ldots,b_n]\big)$.}}]\label{lem:homfly-rec}
Let $[b_1,\ldots,b_n]$ be any even continued fraction where $n \geq 1$. Then
\[ P_0 = l^{-t_n|b_n|} P_2 + (q\half-q\nhalf)\frac{1-l^{-t_n|b_n|}}{l - l\inv} P_1. \]
\end{lemma}

A weaker result from the same paper will be useful in Section~\ref{sec:homfly_spec}.
\begin{lemma}[{\cite[Proof of Lem.\;1]{duzhin_formula_2015}}]\label{lem:homfly-rec-1}
Let $[b_1,\ldots,b_n]$ be any even continued fraction where $n \geq 1$ and $b_n$ may be zero. Then,
\[ P[b_1,\ldots,b_n+2\sgn(b_n)] = l^{-2t_n} P_0 + t_n l^{-t_n} (q\half - q\nhalf) P_1. \]
\end{lemma}

The statement of Theorem~\ref{thm:B} uses the following inductively defined rational function $m[b_1,\ldots,b_n] \in \bbZ(l,q\half)$. It is easy to verify that under 
the specialization from Theorem~\ref{thm:homfly-to-jones}, $m[b_1,\ldots,b_n]$ is exactly the leading term of $V\big(C\big([b_1,\ldots,b_n]\big)\big)$; see Corollary~5.8 in \cite{lee_cluster_2019}.

\begin{definition}\label{def:m}
For any even continued fraction $[b_1,\ldots,b_n]$ where $n > 0$,
\begin{align*}
m[b_1,\ldots,b_n] &= \begin{cases}
-l q\half\, m[b_1,\ldots,b_{n-1}] &\text{if  $t_n = -1$} \\
l^{-|b_n|}\, m[b_1,\ldots,b_{n-2}] &\text{if $t_{n-1} = -1$, $t_n = 1$} \\
l^{1-|b_n|} q\half\, m[b_1,\ldots,b_{n-1}] &\text{if $t_{n-1} = 1$, $t_n = 1$,} \\
\end{cases}
\end{align*}\begin{align*}
m[\,] &= 1, &
m[b_1,\ldots,b_{-1}] &= -l\inv q\nhalf \left( \frac{1-l^2q}{1-q\inv}\right),
\end{align*}
recalling that we defined $t_0 = -1$.
\end{definition}

\begin{remark}\label{rmk:m}
If $n > 1$, $t_{n-1} = -1$, and $t_n = 1$, then
\[ m[b_1,\ldots,b_n] = -l^{-|b_n|-1}q\nhalf\,m[b_1,\ldots,b_{n-1}]. \]
\end{remark}

\section{Path Posets}\label{sec:pps}

In this section we briefly introduce posets and order ideals, then define path posets and give two alternate constructions for the path poset of a continued fraction.

\subsection{Posets}\label{ssec:posets}

A \textbf{poset} is a set with a transitive, antisymmetric, partial relation $<$. Often we will work with the \textbf{Hasse diagram} of a poset, a directed graph on the set of elements of the poset with an edge from $s$ to $t$ if $s < t$ and there exists no $u$ such that $s < u < t$. In lieu of drawing arrows, if $s < t$ we will always place the vertex $t$ vertically above $s$. For a more complete introduction on posets, see \cite{stanley_enumerative_2011}, for example. 
When we care about the particular elements of $P$ rather than just how they relate to one another, we refer to the poset as \textbf{labeled}. Shown in Figure~\ref{fig:Hasse-exs} are some examples of small posets, only the rightmost of which is labeled. 
\begin{figure}[ht]
$\begin{gathered}\begin{tikzpicture}[scale=0.7]
\node (1) at (1,1) {$\bullet$};
\end{tikzpicture}\end{gathered}\quad\quad
\begin{gathered}\begin{tikzpicture}[scale=0.7]
\node (1) at (1,1) {$\bullet$};
\node (2) at (1,1+1*1.414213562) {$\bullet$};
\node (3) at (1,1+2*1.414213562) {$\bullet$};
\draw (1) -- (2) -- (3);
\end{tikzpicture}\end{gathered}\quad\quad
\begin{gathered}\begin{tikzpicture}[scale=0.7]
\node (1) at (1,1) {$\bullet$};
\node (2) at (2,2) {$\bullet$};
\node (3) at (3,1) {$\bullet$};
\draw (1) -- (2) -- (3);
\end{tikzpicture}\end{gathered}\quad\quad
\begin{gathered}\begin{tikzpicture}[scale=0.7*1.414213562]
\node (0) at (0,0) {$\bullet$};
\node (1) at (-1,1) {$\bullet$};
\node (2) at ( 0,1) {$\bullet$};
\node (3) at ( 1,1) {$\bullet$};
\node (4) at (-1,2) {$\bullet$};
\node (5) at ( 0,2) {$\bullet$};
\node (6) at ( 1,2) {$\bullet$};
\node (7) at (0,3) {$\bullet$};
\draw (0) -- (1);
\draw (0) -- (2);
\draw (0) -- (3);
\draw (1) -- (4);
\draw (1) -- (5);
\draw (2) -- (4);
\draw (2) -- (6);
\draw (3) -- (5);
\draw (3) -- (6);
\draw (4) -- (7);
\draw (5) -- (7);
\draw (6) -- (7);
\end{tikzpicture}\end{gathered}\quad\quad
\begin{gathered}\begin{tikzpicture}[scale=0.7]
\node (1) at (1,1) {$b$};
\node (2) at (2,2) {$a$};
\node (3) at (3,1) {$c$};
\node (4) at (2,0) {$d$};
\draw (3) -- (4) -- (1) -- (2) -- (3);\end{tikzpicture}\end{gathered}$
\caption{Hasse diagrams of some small posets.}
\label{fig:Hasse-exs}
\end{figure}
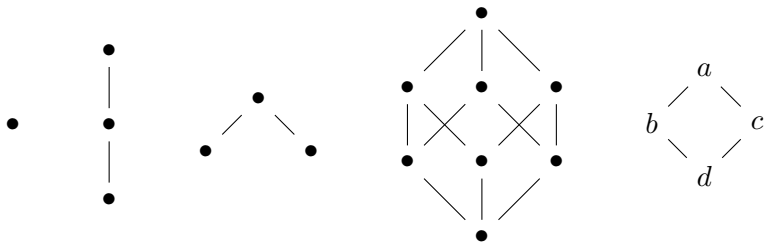

In Section~\ref{sec:cas} we will be interested the following class of sub-posets.

\begin{definition}\label{def:order-ideal}
An \textbf{order ideal} of a poset $P$ is a subset $I \subseteq P$ such that if $t \in I$ and $s < t$, then $s \in I$. We will write this as $I \leq P$.
\end{definition}

\begin{example}
The order ideals of the labeled poset on the right of Figure~\ref{fig:Hasse-exs} are
$\emptyset$, $\{d\}$, $\{b,d\}$, $\{c,d\}$, $\{b,c,d\}$, and $\{a,b,c,d\}$. \end{example}

\subsection{Path Posets}\label{ssec:path_posets}

Path posets were first mentioned, though not named, in \cite{musiker_bases_2013} in relation to perfect matchings of snake graphs. The connection between these posets and perfect matchings is discussed in more detail in \cite{bailey_cluster_2019}, where these posets go by the name `piecewise-linear.' In this section we give a general definition and discuss two different constructions for the path poset of a continued fraction.

A \textbf{path graph} is an undirected graph with $n$ vertices and $n-1$ edges which can be drawn so that all its vertices and edges lie on a single, straight line; see \cite{gross_graph_2005}, for instance. We write $\text{Path}(n)$ to refer to the path graph on $n$ vertices, where $\text{Path}(0)$ is taken to have no vertices or edges. 

\begin{definition}\label{def:path_poset}
A \textbf{path poset} is a poset whose Hasse diagram, viewed as an undirected graph, is $\text{Path}(k)$ for some $k \geq 0$.
\end{definition}

The first three posets on the left of Figure~\ref{fig:Hasse-exs} are path posets. Figure~\ref{fig:path_poset_exs} gives more examples, along with examples of applying the key proposition stated below.

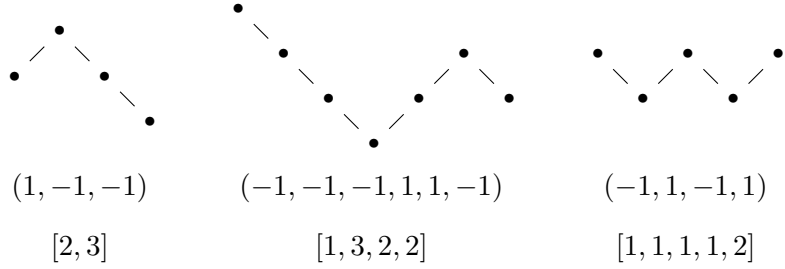
\begin{figure}[ht]
\centering
\begin{tabular}{ccccc}
$\begin{gathered}\begin{tikzpicture}[scale=0.6]
\node (0) at (0,0) {\scalebox{0.8}{$\bullet$}};
\node (1) at (1,1) {\scalebox{0.8}{$\bullet$}};
\node (2) at (2,0) {\scalebox{0.8}{$\bullet$}};
\node (3) at (3,-1) {\scalebox{0.8}{$\bullet$}};
\draw (0) -- (1) -- (2) -- (3);
\end{tikzpicture}\end{gathered}$ &&
$\begin{gathered}\begin{tikzpicture}[scale=0.6]
\node (0) at (0,0) {\scalebox{0.8}{$\bullet$}};
\node (1) at (1,-1) {\scalebox{0.8}{$\bullet$}};
\node (2) at (2,-2) {\scalebox{0.8}{$\bullet$}};
\node (3) at (3,-3) {\scalebox{0.8}{$\bullet$}};
\node (4) at (4,-2) {\scalebox{0.8}{$\bullet$}};
\node (5) at (5,-1) {\scalebox{0.8}{$\bullet$}};
\node (6) at (6,-2) {\scalebox{0.8}{$\bullet$}};
\draw (0) -- (1) -- (2) -- (3) -- (4) -- (5) -- (6);
\end{tikzpicture}\end{gathered}$ &&
$\begin{gathered}\begin{tikzpicture}[scale=0.6]
\node (0) at (0,0) {\scalebox{0.8}{$\bullet$}};
\node (1) at (1,-1) {\scalebox{0.8}{$\bullet$}};
\node (2) at (2,0) {\scalebox{0.8}{$\bullet$}};
\node (3) at (3,-1) {\scalebox{0.8}{$\bullet$}};
\node (4) at (4,0) {\scalebox{0.8}{$\bullet$}};
\draw (0) -- (1) -- (2) -- (3) -- (4);
\end{tikzpicture}\end{gathered}$ \\[14pt]
$(1,-1,-1)$ && $(-1,-1,-1,1,1,-1)$ && $(-1,1,-1,1)$ \\[10pt]
$[2,3]$ && $[1,3,2,2]$ && $[1,1,1,1,2]$
\end{tabular}
\caption{Some (unlabeled) path posets, and as per Propositions~\ref{prop:path_posets_q} and \ref{prop:seqs-pos-cf}, the inner sign sequences and positive continued fractions (with \mbox{$a_n > 1$}) to which they correspond.}
\label{fig:path_poset_exs}
\end{figure}

\begin{proposition}\label{prop:path_posets_q}
There is a bijection between the set of all path posets and the set of all rational numbers $p/q \in \bbQ$ such that $p/q \geq 1$. 
\end{proposition}
\begin{proof}
Given a path poset on $k > 0$ elements, we have $(k-1)$ choices of how to orient each edge in the Hasse diagram. Define a sequence $(s_i) \in \{-1,1\}^{k-1}$ by $s_i = 1$ if the $i^{\text{th}}$ edge in the Hasse diagram is oriented upwards and $s_i = -1$ if it is oriented downwards. Path posets on $k$ elements (up to isomorphism) are in one-to-one correspondence with sign sequences of length $k-1$ in this way. Thus, by Proposition~\ref{prop:seqs-pos-cf}, they are in one-to-one correspondence with positive continued fractions such that $k = \ell_n-1$. 
More generally, Proposition~\ref{prop:seqs-pos-cf} gives us a bijection between the set of all path posets on $k > 0$ elements and the set of all $p/q \in \bbQ$ such that $p/q > 1$. We then say that the degenerate case of the path poset on $\emptyset$ corresponds to $1 \in \bbQ$.
\end{proof}

\begin{definition}\label{def:path_posets_q}
For all $p/q \in \bbQ$ such that $p/q \geq 1$, let $Q(p/q)$ be the path poset corresponding to $p/q$ as per Proposition~\ref{prop:path_posets_q}. We define $Q(\infty) = \emptyset$.
\end{definition}

The following definition, which is more similar in style to Theorem~\ref{thm:rational}, gives a way to realize a broader class of continued fractions as path posets.

\begin{definition}\label{def:path_poset_cf}
Let $[c_1,\ldots,c_n]$ be any continued fraction such that $n \geq 0$, each $|c_i| \geq 1$, and if $t_i = t_{i+1}$ then $|c_i|,|c_{i+1}| > 1$. We define the (labeled) path poset $Q[c_1,\ldots,c_n]$ as follows. We begin with the disjoint union of the posets $S_1,\ldots,S_n$, where each $S_i$ is given by
\[ S_i = \begin{tabular}{ccc}
$\begin{gathered}\begin{tikzpicture}[scale=0.9,xscale=0.8]
\node (lip1) at (1,1) {$\ell_{i-1}+1$};
\node (lip2) at (2,2) {$\ell_{i-1}+2$};
\node (di) at (3,3) {\rotatebox{51.34}{$\cdots$}};
\node (lsim1) at (4,4) {$\ell_i - 1$};
\draw (lip1) -- (lip2) -- (di) -- (lsim1);
\end{tikzpicture}\end{gathered}$ & $\begin{gathered}\emptyset\end{gathered}$
& $\begin{gathered}\begin{tikzpicture}[scale=0.9,xscale=0.8]
\node (lip1) at (1,4) {$\ell_{i-1}+1$};
\node (lip2) at (2,3) {$\ell_{i-1}+2$};
\node (di) at (3,2) {\rotatebox{-51.34}{$\cdots$}};
\node (lsim1) at (4,1) {$\ell_i - 1$};
\draw (lip1) -- (lip2) -- (di) -- (lsim1);
\end{tikzpicture}\end{gathered}$ \\
if $t_i|c_i| > 1$ & if $|c_i| = 1$ & if $t_i|c_i| < -1$
\end{tabular} \]
recalling that $\ell_i = \ell_{i-1} + |c_i|$, and therefore each $S_i$ has $|c_i| - 1$ elements. Then for all $i < n$, if $t_i = 1$ join $S_i$ and $S_{i+1}$ as follows.
\[ \begin{tabular}{ccc}
$\begin{gathered}\begin{tikzpicture}[scale=0.9,xscale=0.8]
\node (lip1) at (1,1) {$\ell_{i-1}+1$};
\node (di) at (2,2) {\rotatebox{51.34}{$\cdots$}};
\node (lsim1) at (3,3) {$\ell_i - 1$};
\node (lsi) at (4,4) {$\textcolor{blue}{\ell_i}$};
\node (lsip1) at (5,3) {$\ell_i+1$};
\node (dsi) at (6,2) {\rotatebox{-51.34}{$\cdots$}};
\node (lsssim1) at (7,1) {$\ell_{i+1} - 1$};
\draw (lip1) -- (di) -- (lsim1);
\draw[color=blue] (lsim1) -- (lsi) -- (lsip1);
\draw (lsip1) -- (dsi) -- (lsssim1);
\end{tikzpicture}\end{gathered}$ & &
$\begin{gathered}\begin{tikzpicture}[scale=0.9,xscale=0.8]
\node (lip1) at (1,1) {$\ell_{i-1}+1$};
\node (di) at (2,2) {\rotatebox{51.34}{$\cdots$}};
\node (lsim1) at (3,3) {$\ell_i - 1$};
\node (lsip1) at (4,2) {$\ell_i+1$};
\node (dsi) at (5,3) {\rotatebox{51.34}{$\cdots$}};
\node (lsssim1) at (6,4) {$\ell_{i+1} - 1$};
\draw (lip1) -- (di) -- (lsim1);
\draw[color=blue] (lsim1) -- (lsip1);
\draw (lsip1) -- (dsi) -- (lsssim1);
\end{tikzpicture}\end{gathered}$ \\
if $t_{i+1} \neq t_i$ &\quad& if $t_{i+1} = t_i$
\end{tabular} \]
Otherwise, if $t_i = -1$, join $S_i$ and $S_{i+1}$ as follows. \[ \begin{tabular}{ccc}
$\begin{gathered}\begin{tikzpicture}[scale=0.9,xscale=0.8]
\node (lip1) at (1,-1) {$\ell_{i-1}+1$};
\node (di) at (2,-2) {\rotatebox{-51.34}{$\cdots$}};
\node (lsim1) at (3,-3) {$\ell_i - 1$};
\node (lsi) at (4,-4) {$\textcolor{blue}{\ell_i}$};
\node (lsip1) at (5,-3) {$\ell_i+1$};
\node (dsi) at (6,-2) {\rotatebox{51.34}{$\cdots$}};
\node (lsssim1) at (7,-1) {$\ell_{i+1} - 1$};
\draw (lip1) -- (di) -- (lsim1);
\draw[color=blue] (lsim1) -- (lsi) -- (lsip1);
\draw (lsip1) -- (dsi) -- (lsssim1);
\end{tikzpicture}\end{gathered}$ & &
$\begin{gathered}\begin{tikzpicture}[scale=0.9,xscale=0.8]
\node (lip1) at (1,-1) {$\ell_{i-1}+1$};
\node (di) at (2,-2) {\rotatebox{-51.34}{$\cdots$}};
\node (lsim1) at (3,-3) {$\ell_i - 1$};
\node (lsip1) at (4,-2) {$\ell_i+1$};
\node (dsi) at (5,-3) {\rotatebox{-51.34}{$\cdots$}};
\node (lsssim1) at (6,-4) {$\ell_{i+1} - 1$};
\draw (lip1) -- (di) -- (lsim1);
\draw[color=blue] (lsim1) -- (lsip1);
\draw (lsip1) -- (dsi) -- (lsssim1);
\end{tikzpicture}\end{gathered}$ \\
if $t_{i+1} \neq t_i$ &\quad& if $t_{i+1} = t_i$
\end{tabular} \]
\end{definition}

An example of this construction is shown in Figure~\ref{fig:path_poset_q_ex}.

\begin{figure}[ht]
\centering
\begin{tikzpicture}[scale=0.8,xscale=0.95]
\node (1) at (1,1) {$1$};
\node (2) at (2,2) {$\textcolor{blue}{2}$};
\node (3) at (3,1) {$3$};
\node (4) at (4,0) {$4$};
\node (6) at (5,1) {$6$};
\node (7) at (6,0) {$7$};
\node (8) at (7,-1) {$8$};
\node (10) at (8,0) {$10$};
\node (11) at (9,-1) {$\textcolor{blue}{11}$};
\node (12) at (10,0) {$12$};
\node (13) at (11,1) {$13$};
\node (14) at (12,2) {$\textcolor{blue}{14}$};
\draw[color=blue] (1) -- (2) -- (3);
\draw (3) -- (4);
\draw[color=blue] (4) -- (6);
\draw (6) -- (7) -- (8);
\draw[color=blue] (8) -- (10) -- (11) -- (12);
\draw (12) -- (13);
\draw[color=blue] (13) -- (14);
\node at (-1,-4) {$t_i|c_i| =$};
\node at (1,1-1-0.75) {$S_1$};
\node at (1,-4) {$2$};
\draw[decorate, decoration={calligraphic brace,amplitude=5pt}, line width=1pt]
    ( $ (1+0.425,1-1) $ ) -- ( $ (1-0.425,1-1) $ );
\node at (3.5,0-1-0.75) {$S_2$};
\node at (3.5,-4) {$-3$};
\draw[decorate, decoration={calligraphic brace,amplitude=5pt}, line width=1pt]
    ( $ (4+0.25,0-1) $ ) -- ( $ (3-0.25,0-1) $ );
\node at (6,-1-1-0.75) {$S_3$};
\node at (6,-4) {$-4$};
\draw[decorate, decoration={calligraphic brace,amplitude=5pt}, line width=1pt]
    ( $ (7+0.25,-1-1) $ ) -- ( $ (5-0.25,-1-1) $ );
\node at (8,-1-1-0.75) {$S_4$};
\node at (8,-4) {$-2$};
\draw[decorate, decoration={calligraphic brace,amplitude=5pt}, line width=1pt]
    ( $ (8+0.425,-1-1) $ ) -- ( $ (8-0.425,-1-1) $ );
\node at (10.5,-1-1-0.75) {$S_5$};
\node at (10.5,-4) {$3$};
\draw[decorate, decoration={calligraphic brace,amplitude=5pt}, line width=1pt]
    ( $ (11+0.25,-1-1) $ ) -- ( $ (10-0.25,-1-1) $ );
\node at (13,1-1-0.75) {$S_6$};
\node at (13,-4) {$-1$};
\node at (13,1-1) {$\emptyset$};
\end{tikzpicture}
\caption{The labeled path poset $Q[2,3,-4,2,3,1]$.}
\label{fig:path_poset_q_ex}
\end{figure}

\begin{figure}[ht]
\centering
\begin{tikzpicture}[scale=0.8,xscale=0.95]
\node (1) at (1,1) {$1$};
\node (2) at (2,2) {$\textcolor{blue}{2}$};
\node (3) at (3,1) {$3$};
\node (4) at (4,0) {$\textcolor{blue}{4}$};
\node (5) at (5,1) {$\textcolor{blue}{5}$};
\node (6) at (6,0) {$6$};
\node (7) at (7,-1) {$\textcolor{blue}{7}$};
\node (8) at (8,0) {$\textcolor{blue}{8}$};
\node (9) at (9,-1) {$\textcolor{blue}{9}$};
\node (10) at (10,0) {$10$};
\node (11) at (11,1) {$11$};
\node (12) at (12,2) {$12$};
\draw[color=blue] (1) -- (2) -- (3) -- (4) -- (5) -- (6) -- (7) -- (8) -- (9) -- (10);
\draw (10) -- (11) -- (12);
\node at (-1,-4) {$t_i|c_i| =$};
\node at (1,1-1-0.75) {$S_1$};
\node at (1,-4) {$2$};
\draw[decorate, decoration={calligraphic brace,amplitude=5pt}, line width=1pt]
    ( $ (1+0.425,1-1) $ ) -- ( $ (1-0.425,1-1) $ );
\node at (3,1-1-0.75) {$S_2$};
\node at (3,-4) {$-2$};
\draw[decorate, decoration={calligraphic brace,amplitude=5pt}, line width=1pt]
    ( $ (3+0.425,1-1) $ ) -- ( $ (3-0.425,1-1) $ );
\node at (4.5,0-1-0.75) {$S_3$};
\node at (4.5,-4) {$1$};
\node at (4.5,0-1) {$\emptyset$};
\node at (6,0-1-0.75) {$S_4$};
\node at (6,-4) {$-2$};
\draw[decorate, decoration={calligraphic brace,amplitude=5pt}, line width=1pt]
    ( $ (6+0.425,0-1) $ ) -- ( $ (6-0.425,0-1) $ );
\node at (7.5,-1-1-0.75) {$S_5$};
\node at (7.5,-4) {$1$};
\node at (7.5,-1-1) {$\emptyset$};
\node at (8.5,-1-1-0.75) {$S_6$};
\node at (8.5,-4) {$-1$};
\node at (8.5,-1-1) {$\emptyset$};
\node at (11,0-1-0.75) {$S_7$};
\node at (11,-4) {$4$};
\draw[decorate, decoration={calligraphic brace,amplitude=5pt}, line width=1pt]
    ( $ (12+0.25,0-1) $ ) -- ( $ (10-0.25,0-1) $ );
\end{tikzpicture}
\caption{The labeled path poset $Q[2,2,1,2,1,1,4]$, isomorphic to $Q[2,3,-4,2,3,1]$ from Figure~\ref{fig:path_poset_q_ex} but with different labels. Both continued fractions are expansions of $206/87$ and so both posets are isomorphic to $Q(206/87)$ by Theorem~\ref{thm:path_poset_cf}.}
\label{fig:path_poset_q_ex_pos}
\end{figure}
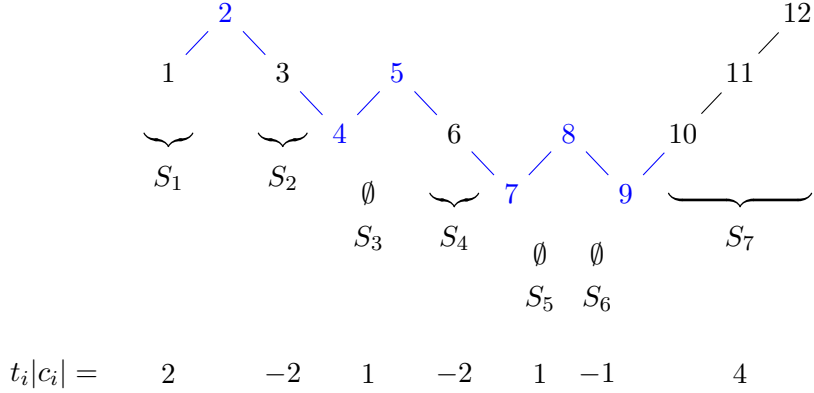

\begin{remark}\label{rmk:path_poset_cf_abs}
A rational number $p/q \in \bbQ^\infty$ has a continued fraction expansion that satisfies the conditions of Definition~\ref{def:path_poset_cf} if and only if $|p/q| \geq 1$. \end{remark}

The follow observation about the order ideals of the sub-posets $S_m$ which appear in Definition~\ref{def:path_poset_cf} will be essential in later sections.

\begin{remark}\label{rmk:Sn-order-ideals}
Every order ideal of $S_m$ is of the following form for some $0 \leq k \leq |c_i|-1$,
where $\lambda_m(j)$, as defined below, is the $j^{\text{th}}$ vertex from the bottom of $S_m$.
\begin{align*}
\left\{ \lambda_m(j) : 1 \leq j \leq k \right\} && \text{where} &&
\lambda_m(j) = \begin{cases}
\ell_{m-1} + j &\text{if $t_m = 1$} \\
\ell_m - j &\text{if $t_m = -1$.}
\end{cases}
\end{align*}
\end{remark}

\begin{example}
If $|c_1| = 2$ and $|c_2| = 5$, meaning $\ell_1 = 3$ and $\ell_2 = 8$, then Remark~\ref{rmk:Sn-order-ideals} states that the order ideals of $S_2$ are exactly:
\begin{align*}
\begin{gathered}
\emptyset, \{4\}, \{4,5\}, \{4,5,6\}, \{4,5,6,7\} \\
\text{if $t_2 = 1$}
\end{gathered} && \begin{gathered}
\emptyset, \{7\}, \{6,7\}, \{5,6,7\}, \{4,5,6,7\} \\
\text{if $t_2 = -1$.}
\end{gathered}
\end{align*}
\end{example}

The following observation about the parity of the labels of each $S_m$ in the case of even continued fractions is the key property of this particular labeling scheme, and is what powers Lemma~\ref{lem:Sn-Fpoly}.

\begin{remark}\label{rmk:Sn-labels}
For even continued fractions, the parity of $\lambda_m(j)$ is the same as the parity of $j$, since each sum $l_i$ is even. As a consequence, for example, the first and last vertices of each $S_m$ are always odd.
\end{remark}

The following theorem shows that, up to the map used in Remark~\ref{rmk:invol}, the two different constructions above for the path poset of a rational number always coincide. As a consequence, it shows that the construction in Definition~\ref{def:path_poset_cf} is well-defined. The proof of this theorem is fairly technical and does not give any insight relevant to understanding the proof of Theorem~\ref{thm:B}, so it is postponed until Appendix~\ref{appx:thm:path_poset_cf}. 
\begin{theorem}\label{thm:path_poset_cf}
For any continued fraction $[c_1,\ldots,c_n]$ which satisfies the conditions of Definition~\ref{def:path_poset_cf} and any $p/q \in \bbQ^\infty$ such that $p/q \geq 1$, if $[c_1,\ldots,c_n] = p/q$ or $p/(q-\sgn(q)\,p)$ then $Q[c_1,\ldots,c_n]$ is isomorphic to $Q(p/q)$.
\end{theorem}
\begin{proof}
See Appendix~\ref{appx:thm:path_poset_cf}.
\end{proof}

It is important to note that the theorem above only provides an isomorphism of posets, and thus the particular labels may change depending on the expansion chosen; see Figures~\ref{fig:path_poset_q_ex} and \ref{fig:path_poset_q_ex_pos} for an example. Using the correspondence between path posets and snake graphs discussed in the next section, this is a key difference between our work and that in \cite{lee_cluster_2019}.

\section{Cluster Algebras and the $F$-Polynomial}\label{sec:cas}

In this section we connect path posets as introduced in Section~\ref{sec:pps} to cluster algebras. We begin by very briefly introducing cluster algebras and the existing work on computing the $F$-polynomial from a snake graph. Then, we discuss the connection between these results and path posets, and define the $F$-polynomial of a path poset. We conclude with some formulae for computing the $F$-polynomial of a path poset.

\subsection{The $F$-Polynomial}\label{ssec:cas}

For our purposes, a \textbf{cluster algebra} $\calA$ is a subalgebra of $\bbQ(x_1^{\pm1},\ldots,x_n^{\pm1},y_1,\ldots,y_n)$ generated by a set $\{x_t\}$ of \textbf{cluster variables}, where each $x_t \in \bbZ[x_1^{\pm1},\ldots,x_n^{\pm1},y_1,\ldots,y_n]$. The \allbf{$F$-polynomial} $F_t \in \bbZ[y_1,\ldots,y_n]$ of the cluster variable $x_t$ is obtained by setting $x_1 = \cdots = x_n = 1$. For a full introduction, including a complete definition of a cluster algebra, see \cite{canakci_cluster_2018}, \cite{schiffler_lecture_2016}, or \cite{fomin_cluster_2002}.

We will be interested in a formula for computing the $F$-polynomial which uses combinatorial objects called snake graphs, first discussed in \cite{musiker_cluster_2010}.

\begin{definition}[\cite{canakci_cluster_2018}]
A \textbf{snake graph} $\calG$ is a planar graph consisting of a sequence of square tiles $G_1,\ldots,G_d$ for $d \geq 1$ such that $G_i$ and $G_{i+1}$ share exactly one edge, and it is either the east edge of $G_i$ identified with the west edge of $G_{i+1}$, or the north edge of $G_i$ identified with the south edge of $G_{i+1}$. A single southern edge is the $d = 0$ case of a snake graph.
\end{definition}

In \cite{canakci_cluster_2018}, the authors show that snake graphs are in one-to-one correspondence with positive continued fractions $[a_1,\ldots,a_n]$, where as in the literature, we write $\calG[a_1,\ldots,a_n]$ to denote the snake graph corresponding to $[a_1,\ldots,a_n]$. 
In \cite{musiker_positivity_2011}, the authors show that for a certain class of cluster algebras, every cluster variable corresponds to a \textbf{labeled snake graph}, a snake graph where each tile $G_i$ has a label $\tau(i) \in \{1,\ldots,n\}$. Furthermore, the authors prove the following cluster expansion formula, which uses the notion of a \textbf{perfect matching} of a snake graph $\calG$, a subset $P$ of the edges of $\calG$ such that every vertex in $\calG$ is adjacent to exactly one edge in $P$. We write $\Match(\calG)$ to denote the set of all perfect matchings of $\calG$, $P_-$ to denote the unique perfect matching of $\calG$ which contains the south edge of $G_1$ and otherwise only boundary edges, and $X \ominus Y := (X \cup Y) \setminus (X \cap Y)$ to denote the symmetric difference.

\begin{theorem}[\cite{musiker_positivity_2011}]\label{thm:expansion-formula}
For any labeled snake graph $\calG$ which corresponds to the cluster variable $x_t \in \calA$,
\[ F_t = \sum_{P \in \Match(\calG)} y(P) \quad\quad\text{ where }\quad\quad y(P) = \prod_{G_i \subseteq P \ominus P_-} y_{\tau(i)}. \]
\end{theorem}

The monomial $y(P)$ is called the \textbf{height} of $P$. We have the following key result, which connects perfect matchings and the height monomial to our work in the previous section.

\begin{theorem}[{\cite[Thm.\;5.4]{musiker_bases_2013}}]\label{thm:perfmatchings}
The perfect matchings of $\calG[a_1,\ldots,a_n]$ 
are in one-to-one correspondence with the order ideals of $Q[a_1,\ldots,a_n]$. Furthermore, if $P$ corresponds to the order ideal $I \leq Q[a_1,\ldots,a_n]$, the height satisfies
\[ y(P) = \prod_{i \in I} y_{\tau(i)}. \]
\end{theorem}

This theorem motivates the following pair of definitions, which together give us a coherent definition for the $F$-polynomial of a path poset. The first, which is a slightly more general version of a similar condition in \cite{lee_cluster_2019}, ensures that the labels of the given path poset are properly realized in an appropriate cluster algebra when computing the $F$-polynomial.

\begin{definition}\label{def:realize}
A cluster variable $x_t \in \calA$ \textbf{realizes} a labeled path poset $Q(p/q)$ if
\begin{itemize}[leftmargin=2.75em]
\item $x_t$ corresponds to a labeled snake graph which corresponds to $p/q$ and
\item $\tau(i)$ is equal to the label of the $i^{\text{th}}$ vertex from the left in the Hasse diagram of $Q(p/q)$, for all $i$.
\end{itemize}
\end{definition}

\begin{definition}\label{def:f-poly}
For any labeled poset $Q$, define
\[ F(Q) := \sum_{I \leq Q} y(I) \quad\quad\text{ where }\quad\quad y(I) := \prod_{i \in I} y_i. \]
\end{definition}

\begin{proposition}\label{prop:F-poly}
If $x_t$ realizes $Q(p/q)$, then $F_t = F(Q(p/q))$.
\end{proposition}

Thus we define $F[c_1,\ldots,c_n] := F(Q[c_1,\ldots,c_n])$. Shown in Figure~\ref{fig:F-poly-exs} is an example of computing the $F$-polynomial of a path poset in this way.

\begin{figure}[ht]
\setlength{\tabcolsep}{1pt}
\begin{tabular}{cccccccccc}
&$\begin{gathered}\begin{tikzpicture}[scale=0.6]
\node (1) at (1,1) {\small \textcolor{gray}{$1$}};
\node (2) at (2,2) {\small \textcolor{gray}{$2$}};
\node (3) at (3,1) {\small \textcolor{gray}{$3$}};
\draw[gray] (1) -- (2) -- (3);
\end{tikzpicture} \end{gathered}$  && $\begin{gathered}\begin{tikzpicture}[scale=0.6]
\node (1) at (1,1) {\small $\textbf{1}$};
\node (2) at (2,2) {\small \textcolor{gray}{$2$}};
\node (3) at (3,1) {\small \textcolor{gray}{$3$}};
\draw[gray] (1) -- (2) -- (3);
\end{tikzpicture} \end{gathered}$ &&
$\begin{gathered}\begin{tikzpicture}[scale=0.6]
\node (1) at (1,1) {\small \textcolor{gray}{$1$}};
\node (2) at (2,2) {\small \textcolor{gray}{$2$}};
\node (3) at (3,1) {\small $\textbf{3}$};
\draw[gray] (1) -- (2) -- (3);
\end{tikzpicture} \end{gathered}$ &&
$\begin{gathered}\begin{tikzpicture}[scale=0.6]
\node (1) at (1,1) {\small $\textbf{1}$};
\node (2) at (2,2) {\small \textcolor{gray}{$2$}};
\node (3) at (3,1) {\small $\textbf{3}$};
\draw[gray] (1) -- (2) -- (3);
\end{tikzpicture} \end{gathered}$ &&
$\begin{gathered}\begin{tikzpicture}[scale=0.6]
\node (1) at (1,1) {\small $\textbf{1}$};
\node (2) at (2,2) {\small $\textbf{2}$};
\node (3) at (3,1) {\small $\textbf{3}$};
\draw[thick] (1) -- (2) -- (3);
\end{tikzpicture} \end{gathered}$ \\
& $\emptyset$ && $\{1\}$ && $\{3\}$ && $\{1,3\}$ && $\{1,2,3\}$ \\ \\
$F[2,2] = $& $1$ &$+$& $y_1$ &$+$& $y_3$ &$+$& $y_1y_3$ &$+$& $y_1y_2y_3$
\end{tabular}
\caption{The order ideals of $Q[2,2]$ and its $F$-polynomial.}
\label{fig:F-poly-exs}
\end{figure}

\subsection{Formulae for the $F$-Polynomial}\label{ssec:f-poly-formulae}

We begin by giving an expression for the $F$-polynomial of the sub-posets $S_m$ defined in Definition~\ref{def:path_poset_cf}. Recall from Remark~\ref{rmk:Sn-order-ideals} that $\lambda_m(j)$ is the label of the $j^{\text{th}}$ vertex from the bottom of $S_m$.

\begin{lemma}\label{lem:F-poly-Sm}
For any continued fraction $[c_1,\ldots,c_n]$ which satisfies the conditions of Definition~\ref{def:path_poset_cf} and $1 \leq m \leq n$,
\[ F(S_m) = \sum_{k=0}^{|c_m|-1} \prod_{j=1}^{k} y_{\lambda_m(j)}. \]
\end{lemma}
\begin{proof}
This follows directly from Definition~\ref{def:f-poly} and Remark~\ref{rmk:Sn-order-ideals}.
\end{proof}

As in \cite{lee_cluster_2019}, fix the following notation for $n \geq 1$.
\[ F_0 = F[c_1,\ldots,c_n] \quad\quad F_1 = F[c_1,\ldots,c_{n-1}] \quad\quad F_2 = F[c_1,\ldots,c_{n-2}] \]

To conclude this section, we give a recursive formula for the $F$-polynomial of a continued fraction.
This is almost identical to Lemma~6.12 in \cite{lee_cluster_2019}, the only difference being the extra $b_1 < 0$ case in \cite{lee_cluster_2019}, an artifact of the special case for $b_1 < 0$ in Definition~6.6 in \cite{lee_cluster_2019}.
Unlike the proof of Lemma~6.12 \cite{lee_cluster_2019}, which depends on technical results from \cite{canakci_snake_2013} and \cite{canakci_snake_2015}, the proof below follows mostly from visual intuition.

First, we fix the following notation for $n > 1$.
\[ Q_0 = Q[c_1,\ldots,c_n] \quad\quad Q_1 = Q[c_1,\ldots,c_{n-1}] \quad\quad Q_2 = Q[c_1,\ldots,c_{n-2}] \]

\begin{lemma}[{\cite[Lemma 6.12]{lee_cluster_2019}}]\label{lem:even-cf-rec}
Let $[c_1,\ldots,c_n]$ be any continued fraction with $n > 1$ which satisfies the conditions of Definition~\ref{def:path_poset_cf}. Then,
\[ F_0 = \begin{cases}
-F_2\prod_{i \in S_n} y_i + F_1 \,F(S_n) &\text{if $t_{n-1} = -1$, $t_n = -1$} \\[0.8em]
F_2\prod_{i \in Q_0 \setminus Q_2} y_i + F_1 \,F(S_n) &\text{if $t_{n-1} = 1$, $t_n = -1$} \\[0.8em]
F_2 + F_1 \,F(S_n)\,y_{l_{n-1}} &\text{if $t_{n-1} = -1$, $t_n = 1$} \\[0.8em]
-F_2\prod_{i \in Q_1 \setminus Q_2} y_i + F_1 \,F(S_n) &\text{if $t_{n-1} = 1$, $t_n = 1$.}
\end{cases} \]
\end{lemma}

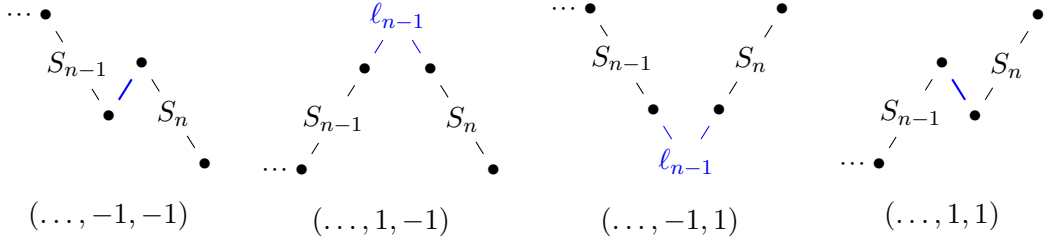
\begin{figure}[ht]
\newcommand{\sh}{0.045}
\newcommand{\tcdots}{$\cdot$\hspace{0.75pt}$\cdot$\hspace{0.75pt}$\cdot$}
\begin{align*}
\begin{gathered}\begin{tikzpicture}[scale=0.67,xscale=0.625]
\node (dim1)    at (0.15,-1)      {\small \tcdots};
\node (lip1)    at (1,-1)         {\small $\bullet$};\node (di)      at (2,-2)         {$S_{n-1}$};\node (lsim1)   at (3,-3)         {\small $\bullet$};\node (lsip1)   at (4+\sh,-2+\sh) {\small $\bullet$};\node (dsi)     at (5+\sh,-3+\sh) {$S_n$};\node (lsssim1) at (6+\sh,-4+\sh) {\small $\bullet$};\draw (lip1) -- (di) -- (lsim1);
\draw[color=blue,thick] (lsim1) -- (lsip1);
\draw (lsip1) -- (dsi) -- (lsssim1);
\end{tikzpicture} \\ (\ldots,-1,-1) \end{gathered} &&
\begin{gathered}\begin{tikzpicture}[scale=0.67,xscale=0.625]
\node (dim1)    at (0.15,1)      {\small \tcdots};
\node (lip1)    at (1,1)         {\small $\bullet$};\node (di)      at (2,2)         {$S_{n-1}$};\node (lsim1)   at (3,3)         {\small $\bullet$};\node (lsi)     at (4+\sh,4+\sh) {$\textcolor{blue}{l_{n-1}}$};         
\node (lsip1)   at (5+\sh+\sh,3) {\small $\bullet$};\node (dsi)     at (6+\sh+\sh,2) {$S_n$};\node (lsssim1) at (7+\sh+\sh,1) {\small $\bullet$};\draw (lip1) -- (di) -- (lsim1);
\draw[color=blue] (lsim1) -- (lsi) -- (lsip1);
\draw (lsip1) -- (dsi) -- (lsssim1);
\end{tikzpicture} \\ (\ldots,1,-1) \end{gathered} &&
\begin{gathered}\begin{tikzpicture}[scale=0.67,xscale=0.625]
\node (dim1)    at (0.15,-1)      {\small \tcdots};
\node (lip1)    at (1,-1)         {\small $\bullet$};\node (di)      at (2,-2)         {$S_{n-1}$};\node (lsim1)   at (3,-3)         {\small $\bullet$};\node (lsi)     at (4+\sh,-4-\sh) {$\textcolor{blue}{l_{n-1}}$};         
\node (lsip1)   at (5+\sh+\sh,-3) {\small $\bullet$};\node (dsi)     at (6+\sh+\sh,-2) {$S_n$};\node (lsssim1) at (7+\sh+\sh,-1) {\small $\bullet$};\draw (lip1) -- (di) -- (lsim1);
\draw[color=blue] (lsim1) -- (lsi) -- (lsip1);
\draw (lsip1) -- (dsi) -- (lsssim1);
\end{tikzpicture} \\ (\ldots,-1,1) \end{gathered} &&
\begin{gathered}\begin{tikzpicture}[scale=0.67,xscale=0.625]
\node (dim1)    at (0.15,1)      {\small \tcdots};
\node (lip1)    at (1,1)         {\small $\bullet$};\node (di)      at (2,2)         {$S_{n-1}$};\node (lsim1)   at (3,3)         {\small $\bullet$};\node (lsip1)   at (4+\sh,2-\sh) {\small $\bullet$};\node (dsi)     at (5+\sh,3-\sh) {$S_n$};\node (lsssim1) at (6+\sh,4-\sh) {\small $\bullet$};\draw (lip1) -- (di) -- (lsim1);
\draw[color=blue,thick] (lsim1) -- (lsip1);
\draw (lsip1) -- (dsi) -- (lsssim1);
\end{tikzpicture} \\ (\ldots,1,1) \end{gathered}
\end{align*}
\caption{The four ways $S_{n-1}$ and $S_n$ can be connected, with $\type[c_1,\ldots,c_n]$ shown below.}
\label{fig:poset-conns}
\end{figure}

\begin{proof}
Recall from Definition~\ref{def:f-poly} that $F_0$ is the sum of the heights of all order ideals $I \leq Q_0$. We will also use the fact that over a disjoint union $F(P \cup P') = F(P)\,F(P')$.

Suppose $t_{n-1} = 1$ and $t_n = -1$. The sub-posets $S_{n-1}$ and $S_n$ are connected as shown in the second picture from the left in Figure~\ref{fig:poset-conns}. Any order ideal $I \leq Q_0$ which contains $l_{n-1}$ must also contain everything not in $Q_2$, meaning all of $S_{n-1}$ and $S_n$, as well as any connecting vertex that may exist between $S_{n-1}$ and $Q_2$. Since the remainder of $I$ can be any order ideal of $Q_2$, the sum of the heights of all such order ideals is exactly $F_2\prod_{i \in Q_0 \setminus Q_2} y_i$. This is the first term in the desired expression for $F_0$. Any order ideal $I \leq Q_0$ which does not contain $l_{n-1}$ must be the union of some order ideal of $Q_1$ and some order ideal of $S_n$. Thus, the sum over the heights of these order ideals contributes the remaining $F_1\,F(S_n)$ term.

The case of $t_{n-1} = -1$ and $t_n = 1$, shown in the third picture form the left in Figure~\ref{fig:poset-conns}, is similar. Any order ideal $I \leq Q_0$ which does not contain $l_{n-1}$ cannot contain any elements of $S_{n-1}$, $S_n$, or any connecting vertex that may exist between $S_{n-1}$ and $Q_2$. Thus $I$ must be an order ideal of $Q_2$, and can be any order ideal of $Q_2$. Thus the sum of the heights of all such order ideals contributes the $F_2$ term in the desired expression for $F_0$. Any order ideal $I \leq Q_0$ which does contain $l_{n-1}$ otherwise consists of the union of some order ideal of $Q_1$ and some order ideal of $S_n$. This contributes the remaining $F_1\,F(S_n)\,y_{l_{n-1}}$ term.

Suppose $t_{n-1} = -1$ and $t_n = -1$, shown in the leftmost picture in Figure~\ref{fig:poset-conns}. Let $Q_0'$ be the poset $Q_0$ without the highlighted relation connecting $S_{n-1}$ and $S_n$. The only order ideals of $Q_0'$ which are not order ideals of $Q_0$ are those which contain the topmost vertex of $S_n$ (and therefore all of $S_n$) and do not contain the bottom-most vertex of $S_{n-1}$ (and therefore none of $S_{n-1}$). The sum of the heights of all such order ideals is exactly $F_2\prod_{i \in S_n} y_i$; subtracting this from $F(Q_0') = F(Q_1)\,F(S_n)$ gives us the desired expression for $F_0$. 

The case of $t_{n-1} = 1$ and $t_n = 1$, shown in the rightmost picture in Figure~\ref{fig:poset-conns}, is similar. The only order ideals of $Q_0'$ which are not order ideals of $Q_0$ are those which contain the topmost vertex of $S_{n-1}$ (and therefore all of $Q_1 \setminus Q_2$) and do not contain the bottom-most vertex of $S_n$ (and therefore none of $S_n$). Following the previous case, this gives us the desired expression for $F_0$.
\end{proof}

\section{Main Results}\label{sec:homfly_spec}

This section contains the proof of Theorem~\ref{thm:B}. This proof has two instances of induction, one on the length of the continued fraction $n$ and one on $|b_1|$ in the base case of $n = 1$. The inductive step on $n$ reduces to mechanically checking the compatability of Lemma~\ref{lem:homfly-rec}, Definition~\ref{def:m}, and Corollary~\ref{cor:FP-rec} in various cases of the type sequence. Lemma~\ref{lem:Sn-Fpoly} powers Corollary~\ref{cor:FP-rec}, and by extension, this section of the proof. The most interesting part of the main proof is the base case of $|b_1| = 2$ and $n = 1$ (and to a lesser extent, the inductive step on $|b_1|$ when $b_1 > 0$) where due to the fact that the fraction $P[b_1,\ldots,b_{-1}]$ cannot be written as a polynomial, the strange choice of specialization for $y_1$ is needed.

\subsection{A Specialization for the HOMFLY Polynomial}\label{ssec:homfly_spec}

As in the statement of Theorem~\ref{thm:B}, we will consider the following specialization of the $F$-polynomial.
\[ \tilde{F}[b_1,\ldots,b_n] := \varphi_P(F[b_1,\ldots,b_n]) \in \bbZ(l,q\half) \enspace\text{ where} \]\begin{align*}
\varphi_P(y_1) &= l^2\left(\frac{1 - l^2q}{1 - q\inv}\right)\inv \text{ if $b_1 > 0$} \enspace\text{ or }\enspace q^{-2}\left(\frac{1 - l^2q}{1 - q\inv}\right) \text{ if $b_1 < 0$,}
\end{align*}
\begin{align*}
\varphi_P(y_{2i}) &= -l^2q, & \varphi_P(y_{2i+1}) &= -q\inv & &\text{for all } i \geq 1.
\end{align*}

Under the substitutions in Theorem~\ref{thm:homfly-to-jones}, this is exactly the specialization of the Jones polynomial given in \cite{lee_cluster_2019}. In particular, the term
\[ w := \frac{1-l^2q}{1-q\inv} = \frac{1+\varphi_P(y_{\text{even}})}{1+\varphi_P(y_{\text{odd}})} \quad\quad\text{ becomes }\quad\quad \frac{1-t\inv}{1-t\inv} = 1. \]

First, as in Section~\ref{ssec:homfly_formulae}, we fix the following notation for $n > 1$.
\[ \tilde{F}_0 = \tilde{F}[a_1,\ldots,a_n] \quad\quad \tilde{F}_1 = \tilde{F}[a_1,\ldots,a_{n-1}] \quad\quad \tilde{F}_2 = \tilde{F}[a_1,\ldots,a_{n-2}] \]
We will also use the following notation, which can be found in \cite{lee_cluster_2019}.
\begin{equation*}
[k]_x = \frac{1-x^k}{1-x} = 1 + x + x^2 + \cdots + x^{k-1}
\end{equation*}

Our first goal will be to apply our specialization $\varphi_P$ to Lemma~\ref{lem:even-cf-rec}. In order to do so we will need to know the value of $\varphi_P$ applied to the expressions $F(S_m)$ and $\prod_{i \in S_m} y_i$ which appear in Lemma~\ref{lem:even-cf-rec}. This is established by following lemma. 
As stated earlier, is it this lemma that entirely depends on the nice behavior of the labels in $Q[b_1,\ldots,b_n]$ (Remark~\ref{rmk:Sn-labels}).

\begin{lemma}\label{lem:Sn-Fpoly}
For any $[b_1,\ldots,b_n]$ and $1 < m \leq n$,
\begin{align*}
\text{(i)} && \tilde{F}(S_m) &= (1-q\inv) \big[|b_m|/2\big]_{l^2} = (1-q\inv)(1 + l^2 + \cdots + l^{|b_m|-2}), \\
\text{(ii)} && \prod_{i \in S_m} \varphi_P(y_i) &= -l^{|b_m|-2}q\inv.
\end{align*}
\end{lemma}
\begin{proof}
Since $|b_1| \geq 2$ and $m > 1$, we have $l_{m-1} > 1$ and therefore $\lambda_m(j) > 1$ for all $1 \leq j \leq |b_m|-1$. Remark~\ref{rmk:Sn-labels} tells us that the parity of $\lambda_m(j)$ is always the same as the parity of $j$, and so we have the following for all $i$.
\begin{equation}\label{eq:lem:Sn-Fpoly_1}
\prod_{j=1}^{i+1} \varphi_P(y_{\lambda_m(j)}) = \begin{cases}
-l^2q \displaystyle\prod_{j=1}^{i} \varphi_P(y_{\lambda_m(j)}) &\text{if $i$ odd}\\
-q\inv \displaystyle\prod_{j=1}^{i} \varphi_P(y_{\lambda_m(j)}) &\text{if $i$ even.}
\end{cases}
\end{equation}
Iterating this once gives us the following, regardless of the parity of $i$.
\begin{equation}\label{eq:lem:Sn-Fpoly_2}
\prod_{j=1}^{i+2} \varphi_P(y_{\lambda_m(j)}) = l^2 \prod_{j=1}^{i} \varphi_P(y_{\lambda_m(j)})
\end{equation}
Thus, by iterating \eqref{eq:lem:Sn-Fpoly_2} we get the following for any $k$.
\begin{align*}
\prod_{j=1}^{2k+1} \varphi_P(y_{\lambda_m(j)}) = -l^{2k}q\inv &&
\prod_{j=1}^{2k} \varphi_P(y_{\lambda_m(j)}) = l^{2k},
\end{align*}
which when used with Lemma~\ref{lem:F-poly-Sm}, give us that
\begin{align*}
\tilde{F}(S_m) = \sum_{i=0}^{|b_m|-1} \prod_{j=1}^i \varphi_P(y_{\lambda_m(j)})
&= \sum_{k=0}^{|b_m|/2-1} \left( \prod_{j=1}^{2k} \varphi_P(y_{\lambda_m(j)}) + \prod_{j=1}^{2k+1} \varphi_P(y_{\lambda_m(j)}) \right) \\
&= \sum_{k=0}^{|b_m|/2-1} (1-q\inv)\,l^{2k} = (1-q\inv) \big[|b_m|/2\big]_{l^2}
\end{align*}
\begin{align*}
\text{and }\;\;& \prod_{i \in S_m}\varphi_P(y_i) = \prod_{j=1}^{|b_m|-1} \varphi_P(y_{\lambda_m(j)}) = -\ell^{|b_m|-2}q\inv. \qedhere
\end{align*}
\end{proof}

\begin{remark}\label{rmk:S1-Fpoly}
If $b_1 > 0$, $\lambda_1(j) = j$. Thus equation \eqref{eq:lem:Sn-Fpoly_2} still holds when $m = 1$, but only for $i > 0$. Iterating \eqref{eq:lem:Sn-Fpoly_2} and using the fact that $\varphi_P(y_{\lambda_1(1)}) = l^2w\inv$, we get that in this case
\[ \prod_{i \in S_1} \varphi_P(y_i) = l^{|b_1|}w\inv. \]
\end{remark}

We can now apply our specialization to Lemma~\ref{lem:even-cf-rec}. The proof of this corollary is completely mechanical.

\begin{corollary}\label{cor:FP-rec}
Let $[b_1,\ldots,b_n]$ be any even continued fraction with $n > 1$. Then
\[ \tilde{F}_0 = \mu(l,q)\, \tilde{F}_2 + \nu(l,q)\, (1 - q\inv) \big[|b_n|/2\big]_{l^2}\, \tilde{F}_1, \]
where the values of $\mu(l,q)$ and $\nu(l,q)$ depend on $\type[b_1,\ldots,b_n]$ as given in the following table.
\begin{center}
{\def\arraystretch{1.4}
\begin{tabular}{|c|c|r|}
\hline $\mu(l,q)$ & $\nu(l,q)$ & $\type[b_1,\ldots,b_n]$ \\\hline
$l^{|b_n|-2} q\inv $ & $1$ &$(\ldots,-1,-1)$ \\
$l^{|b_2|+|b_1|} w\inv $ & $1$ &$(1,-1)$ \\
$l^{|b_n|+|b_{n-1}|} $ & $1$ &$(\ldots,-1,1,-1)$ \\
$-l^{|b_n|+|b_{n-1}|-2}q\inv $ & $1$ &$(\ldots,1,1,-1)$ \\
$1$ & $-l^2q$ &$(\ldots,-1,1)$ \\
$-l^{|b_1|}w\inv$ & $1$ &$(1,1)$ \\
$-l^{|b_{n-1}|}$ & $1$ &$(\ldots,-1,1,1)$ \\
$l^{|b_{n-1}|-2}q\inv$ & $1$ &$(\ldots,1,1,1)$ \\\hline
\end{tabular}}\end{center}
\end{corollary}
\begin{proof}
By Lemma~\ref{lem:even-cf-rec} and Lemma~\ref{lem:Sn-Fpoly}~(i), we only need to show that $\mu(l,q)$ and $\nu(l,q)$ are correctly defined. In the type $(\ldots,-1,1)$ case, this follows from the fact that each $l_i$ is even and therefore $\varphi_P(y_{l_{n-1}}) = -l^2q$. In the remaining cases, we only need to check $\mu(l,q)$.

In type $(\ldots,-1,-1)$, we need to check that $-\prod_{i\in S_n} \varphi_P(y_i) = l^{|b_n|-2}q\inv$. This is exactly Lemma~\ref{lem:Sn-Fpoly}~(ii).

In type $(1,-1)$, we need to check that $\prod_{i\in Q_0\setminus Q_2} \varphi_P(y_i) = l^{|b_2|+|b_1|}w\inv$. Since $t_1 \neq t_2$, there is a vertex $l_1$ connecting $S_1$ and $S_2$. Thus, $Q_0 \setminus Q_2 = S_1 \cup \{l_1\} \cup S_2$, and so by Remark~\ref{rmk:S1-Fpoly} and Lemma~\ref{lem:Sn-Fpoly}~(ii),
\begin{align*}
\prod_{i\in Q_0\setminus Q_2} \varphi_P(y_i)
&= \left( \prod_{i\in S_1} \varphi_P(y_i) \right) \varphi_P(y_{l_1}) \left( \prod_{i\in S_2} \varphi_P(y_i) \right) \\
&= \left( l^{|b_1|} w\inv \right) (-l^2q) \left( -l^{|b_2|-2}q\inv \right) \\ &= l^{|b_2|+|b_1|}w\inv.
\end{align*}

In type $(\ldots,-1,1,-1)$, we need to check that $\prod_{i\in Q_0\setminus Q_2} \varphi_P(y_i) = l^{|b_n|+|b_{n-1}|}$. Since $t_{n-1} \neq t_n$, there is a vertex $l_n$ connecting $S_{n-1}$ and $S_n$, and since $t_{n-2} \neq t_{n-1}$, there is a vertex $l_{n-1}$ connecting $S_{n-2}$ and $S_{n-1}$. Thus, $Q_0 \setminus Q_2 = \{l_{n-2}\} \cup S_{n-1} \cup \{l_{n-1}\} \cup S_{n}$, and so by Lemma~\ref{lem:Sn-Fpoly}~(ii),
\begin{align*}
\prod_{i\in Q_0\setminus Q_2} \varphi_P(y_i)
&= \varphi_P(y_{l_{n-2}}) \left( \prod_{i\in S_{n-1}} \varphi_P(y_i) \right) \varphi_P(y_{l_{n-1}}) \left( \prod_{i\in S_{n}} \varphi_P(y_i) \right) \\
&= (-l^2q) \left( -l^{|b_{n-1}|-2}q\inv \right) (-l^2q) \left( -l^{|b_{n}|-2}q\inv \right) \\ &= l^{|b_n|+|b_{n-1}|}.
\end{align*}

In type $(\ldots,1,1,-1)$, we need to check that $\prod_{i\in Q_0\setminus Q_2} \varphi_P(y_i) = -l^{|b_n|+|b_{n-1}|-2}q\inv$. Since $t_{n-1} \neq t_n$, there is a vertex $l_n$ connecting $S_{n-1}$ and $S_n$, and since $t_{n-2} = t_{n-1}$, there is no vertex connecting $S_{n-2}$ and $S_{n-1}$. Thus, $Q_0 \setminus Q_2 = S_{n-1} \cup \{l_{n-1}\} \cup S_{n}$, and so by Lemma~\ref{lem:Sn-Fpoly}~(ii),
\begin{align*}
\prod_{i\in Q_0\setminus Q_2} \varphi_P(y_i)
&= \left( \prod_{i\in S_{n-1}} \varphi_P(y_i) \right) \varphi_P(y_{l_{n-1}}) \left( \prod_{i\in S_{n}} \varphi_P(y_i) \right) \\
&= \left( -l^{|b_{n-1}|-2}q\inv \right) (-l^2q\inv) \left( -l^{|b_{n}|-2}q\inv \right) \\ &= -l^{|b_n|+|b_{n-1}|-2}q\inv.
\end{align*}

In type $(1,1)$, we need to check that $-\prod_{i\in Q_1\setminus Q_2} \varphi_P(y_i) = -l^{|b_1|}w\inv$. Since in this case $Q_1 \setminus Q_2 = S_1$, this is exactly Remark~\ref{rmk:S1-Fpoly}.

In type $(\ldots,-1,1,1)$, we need to check that $-\prod_{i\in Q_1\setminus Q_2} \varphi_P(y_i) = -l^{|b_{n-1}|}$. Since $t_{n-2} \neq t_{n-1}$, there is a vertex $l_{n-1}$ connecting $S_{n-2}$ and $S_{n-1}$. Thus, $Q_1 \setminus Q_2 = \{l_{n-2}\} \cup S_{n-1}$, and so by Lemma~\ref{lem:Sn-Fpoly}~(ii),
\begin{align*}
-\prod_{i\in Q_1\setminus Q_2} \varphi_P(y_i)
&= -\varphi_P(y_{l_{n-2}}) \left( \prod_{i\in S_{n-1}} \varphi_P(y_i) \right)  \\
&= -(-l^2q) \left( -l^{|b_{n-1}|-2}q\inv \right) \\
&= -l^{|b_{n-1}|}.
\end{align*}

In type $(\ldots,1,1,1)$, we need to check that $\prod_{i\in Q_1\setminus Q_2} \varphi_P(y_i) = l^{|b_{n-1}|-2}q\inv$. Since $t_{n-2} = t_{n-1}$, there is no vertex connecting $S_{n-2}$ and $S_{n-1}$. Thus, $Q_1 \setminus Q_2 = S_{n-1}$, and so this is exactly Lemma~\ref{lem:Sn-Fpoly}~(ii).
\end{proof}

For the inductive step of the $n = 1$ case of Theorem~\ref{thm:B}, we will also need the following fact.

\begin{lemma}\label{lem:FP-rec-1}
For any even $b_1 \neq 0$,
\[ \tilde{F}[b_1 + 2\,\sgn(b_1)] = \begin{cases}
\tilde{F}[b_1] + l^{b_1+2}w\inv(1-q) &\text{ if $b_1 > 0$} \\
l^2\,\tilde{F}[b_1] + 1 - q\inv &\text{ if $b_1 < 0$.}
\end{cases} \]
\end{lemma}
\begin{proof}
First suppose $b_1 > 0$. It follows from Lemma~\ref{lem:F-poly-Sm} that
\begin{align*}
F[b_1+2] &= F[b_1] + \textstyle\prod_{i=1}^{b_1} y_i + \prod_{i=1}^{b_1+1} y_i \\
         &= F[b_1] + \textstyle\left(\prod_{i=1}^{b_1+1} y_i\right)(1 + y_{b_1+1}\inv).
\end{align*}
The fact that $\varphi_P\left(\textstyle\prod_{i=1}^{b_1+1} y_i\right) = l^{b_1+2}w\inv$ (Remark~\ref{rmk:S1-Fpoly}) and $(b_1+1) > 1$ is odd gives us the desired equality. 
If $b_1 < 0$, it follows from Lemma~\ref{lem:F-poly-Sm} that
\[ F[b_1-2] = F[b_1]\,y_{|b_1|}\, y_{|b_1-1|} + y_{|b_1-1|} + 1. \]
The fact that $|b_1| > 1$ is even and $|b_1-1| > 1$ is odd gives us the desired equality.
\end{proof}

We are now ready to prove Theorem~\ref{thm:B}, which equates the HOMFLY polynomial $P[b_1,\ldots,b_n]$ to the specialized $F$-polynomial scaled by $m[b_1,\ldots,b_n]$ (Definition~\ref{def:m}). We restate this theorem below.

{\renewcommand{\thetheorem}{\ref{thm:B}}
\begin{theorem}
For all even continued fractions $[b_1,\ldots,b_n]$ with $n \geq 0$,
\[ P[b_1,\ldots,b_n] = m[b_1,\ldots,b_n]\, \tilde{F}[b_1,\ldots,b_n]. \]
\end{theorem}}
\begin{proof}
We will proceed by induction on $n$.

If $n = 0$, indeed $P[] = 1 = m[]\,\tilde{F}[]$.

If $n = 1$, we will show $P[b_1] = m[b_1]\,\tilde{F}[b_1]$ by induction on $|b_1|/2$. For the base case $|b_1| = 2$, we have two cases. First, if $b_1 > 0$ we have
\begin{align*}
P[2]
&= l^{-2} P[b_1,\ldots,b_{-1}] + l\inv (q\half - q\nhalf)
    &&\text{by Lemma~\ref{lem:homfly-rec-1}}\\
&= l^{-2} \left(\frac{l-l\inv}{q\half - q\nhalf}\right) + l\inv q\half - l\inv q\nhalf
    &&\text{by Corollary~\ref{cor:homfly-unknot}}\\
&= - l^{-3} q\nhalf \left( \frac{1 - l^2}{1 - q\inv} - l^2 q + l^2 \right)
    &&\text{factor out $-l^{-3}q\nhalf$}\\
&= - l^{-3} q\nhalf \left( w + l^2 \right)
    &&\text{by \eqref{eq1} below}\\
&= -l^{-3} q\nhalf w \left(1 + l^2w\inv \right)
    &&\text{factor out $w$}\\
&= m[2] \left(1 + l^2w\inv \right)
    &&\text{by \eqref{eq1second} below}
\end{align*}
where we use the equalities
\begin{gather}
\frac{1-l^2}{1-q\inv} - l^2q = \frac{1 - l^2 - l^2q + l^2}{1-q\inv} = \frac{1-l^2q}{1-q\inv} = w \label{eq1} \\
m[2] = l^{-2}m[b_1,\ldots,b_{-1}] = -l^{-3}q\nhalf w \label{eq1second}
\end{gather}
the latter of which follows from Definition~\ref{def:m}. 
The fact that $F[2] = 1 + y_1$, and therefore $\tilde{F}[2] = 1 + l^2w\inv$, gives us that $P[2] = m[2]\,\tilde{F}[2]$. 
Next, if $b_1 < 0$ we have
\begin{align*}
P[-2]
&= l^2 P[b_1,\ldots,b_{-1}] - l(q\half - q\nhalf)
    &&\text{by Lemma~\ref{lem:homfly-rec-1}}\\
&= l^2 \left(\frac{l-l\inv}{q\half - q\nhalf}\right) - l q\half + l q\nhalf
    &&\text{by Corollary~\ref{cor:homfly-unknot}}\\
&= - l q\half \left( q\inv \frac{1 - l^2}{1 - q\inv} + 1 - q\inv \right)
    &&\text{factor out $-lq\half$}\\
&= - l q\half \left( 1 + q^{-2} w \right)
    &&\text{by \eqref{eq2} below}\\
&= m[-2] \left( 1 + q^{-2} w \right)
    &&\text{Definition~\ref{def:m}}
\end{align*}
where we use the equality
\begin{equation}\label{eq2}
q\inv \frac{1-l^2}{1-q\inv} - q\inv = \frac{q\inv - l^2 q\inv - q\inv + q^{-2}}{1-q\inv} = q^{-2}\frac{1-l^2q}{1-q\inv} = q^{-2} w.
\end{equation}
The fact that $F[-2] = 1 + y_1$, and therefore $\tilde{F}[-2] = 1 + q^{-2}w$, gives us that $P[-2] = m[-2]\,\tilde{F}[-2]$. 

Now suppose $P[b_1] = m[b_1]\,\tilde{F}[b_1]$ for some even nonzero $b_1$. We again have two cases. If $b_1 > 0$, we want to show that $P[b_1+2] = m[b_1+2]\,\tilde{F}[b_1+2]$. Indeed
\begin{align*}
P[b_1+2]
&= l^{-2} P[b_1] + l\inv (q\half - q\nhalf)
    &&\text{by Lemma~\ref{lem:homfly-rec-1}}\\
&= l^{-2} m[b_1]\,\tilde{F}[b_1] + l\inv (q\half - q\nhalf)
    &&\text{by induction}\\
&= -l^{-3-b_1} q\nhalf w\,\tilde{F}[b_1] + l\inv (q\half - q\nhalf)
    &&\text{by Definition~\ref{def:m}}\\
&= -l^{-3-b_1} q\nhalf w \left( \tilde{F}[b_1] + l^{b_1+2} w\inv (1-q) \right)
    &&\text{factor out $-l^{-3-b_1} q\nhalf w$}\\
&= m[b_1+2]\,\tilde{F}[b_1+2]
    &&\text{by Def.~\ref{def:m}, Lem.~\ref{lem:FP-rec-1}.}
\end{align*}
If $b_1 < 0$ we want to show that $P[b_1-2] = m[b_1-2]\tilde{F}[b_1-2]$. Indeed
\begin{align*}
P[b_1-2]
&= l^2 P[b_1] - l (q\half - q\nhalf)
    &&\text{by Lemma~\ref{lem:homfly-rec-1}}\\
&= l^2 m[b_1]\,\tilde{F}[b_1] - l (q\half - q\nhalf)
    &&\text{by induction}\\
&= -l q\half l^2 \tilde{F}[b_1] - l (q\half - q\nhalf)
    &&\text{by Definition~\ref{def:m}}\\
&= -l q\half \left( l^2 \tilde{F}[b_1] + 1 - q\inv \right).
    &&\text{factor out $-lq\half$}\\
&= m[b_1-2]\,\tilde{F}[b_1-2]
    &&\text{by Def.~\ref{def:m}, Lem.~\ref{lem:FP-rec-1}.}
\end{align*}
This completes the base case of $n = 1$.

Now suppose $n > 1$. Recall from Lemma~\ref{lem:homfly-rec},
\[ P_0 = l^{-t_n|b_n|}\,P_2 + (1 - q\inv) \,q\half\, \frac{1-l^{-t_n|b_n|}}{l-l\inv} P_1. \]
By induction, we have that $P_1 = m_1 \tilde{F}_1$ and $P_2 = m_2 \tilde{F}_2$, where we define
\[ m_0 = m[a_1,\ldots,a_n] \quad\quad m_1 = m[a_1,\ldots,a_{n-1}] \quad\quad m_2 = m[a_1,\ldots,a_{n-2}]. \]
Thus,
\begin{equation}\label{eq:main2_1}
P_0 = l^{-t_n|b_n|}\,m_2\, \tilde{F}_2 + (1 - q\inv) \,q\half\, \frac{1-l^{-t_n|b_n|}}{l-l\inv}\,m_1\, \tilde{F}_1.
\end{equation}
On the other hand, by Corollary~\ref{cor:FP-rec},
\begin{equation}\label{eq:main2_2}
m_0\,\tilde{F}_0 = \mu(l,q)\,m_0\, \tilde{F}_2 + (1-q\inv) \,\nu(l,q)\, [|b_n|/2]_{\ell^2}\,m_0\,\tilde{F}_1.
\end{equation}
where $\mu(l,q)$ and $\nu(l,q)$ depend on $\type[b_1,\ldots,b_n]$. Thus, to conclude that $P_0 = m_0\,\tilde{F}_0$, it suffices to show the equality of the right hand sides of \eqref{eq:main2_1} and \eqref{eq:main2_2}. As per the definitions of $\mu(l,1)$ and $\nu(l,q)$, we have eight cases to check. This will be completely mechanical.

First, if $t_n = -1$, by Corollary~\ref{cor:FP-rec}, $\nu(l,q) = 1$, and by Definition~\ref{def:m}, $m_0 = -lq\half m_1$. 
Therefore,
\[ q\half \frac{1-l^{|b_n|}}{l-l\inv}\,m_1 = [|b_n|/2]_{l^2}(-lq\half)\,m_1 = \nu(l,q)\, [|b_n|/2]_{l^2}\, m_0 \]
since $(1-l^{|b_n|})/(l-l\inv) = -l(1-l^{|b_n|})/(1-l^2) = -l[|b_n|/2]_{l^2}$. Thus the coefficients of $\tilde{F}_1$ in \eqref{eq:main2_1} and \eqref{eq:main2_2} are equal. For the coefficients of $\tilde{F}_2$ we still have four cases.

In type $(\ldots,-1,-1)$, by Corollary~\ref{cor:FP-rec} and Definition~\ref{def:m},
\begin{align*}
\mu(l,q) &= l^{|b_n|-2}q\inv &
m_0 &= (-lq\half) (-lq\half m_2) = l^2 q\,m_2
\end{align*}
and therefore
\begin{gather*}
l^{|b_n|}\,m_2 = (l^{|b_n|-2}q\inv) (l^2 q)\, m_2 = \mu(l,q)\, m_0.\end{gather*}

In type $(1,-1)$, by Corollary~\ref{cor:FP-rec} and Definition~\ref{def:m},
\begin{align*}
\mu(l,q) &= l^{|b_2|+|b_1|} w\inv &
m_0 &= (-lq\half)(-l^{-|b_1|-1} q\nhalf w\,m_2) = l^{-|b_1|} w\, m_2
\end{align*}
and therefore
\begin{gather*}
l^{|b_2|}\,m_2 = (l^{|b_2|+|b_1|} w\inv)(l^{-|b_1|} w)\, m_2 = \mu(l,q)\, m_0.\end{gather*}

In type $(\ldots,-1,1,-1)$, by Corollary~\ref{cor:FP-rec}, Definition~\ref{def:m}, and Remark~\ref{rmk:m},
\begin{align*}
\mu(l,q) &= l^{|b_n|+|b_{n-1}|} &
m_0 &= (-lq\half)(-l^{-|b_{n-1}|-1} q\nhalf m_2) = l^{-|b_{n-1}|}\, m_2
\end{align*}
and therefore
\begin{gather*}
l^{|b_n|}\,m_2 = l^{|b_n|+|b_{n-1}|}\,l^{-|b_{n-1}|}\, m_2 = \mu(l,q)\, m_0.\end{gather*}

In type $(\ldots,1,1,-1)$, by Corollary~\ref{cor:FP-rec} and Definition~\ref{def:m}
\begin{align*}
\mu(l,q) &= -l^{|b_n|+|b_{n-1}|-1}q\inv &
m_0 &= (-lq\half)(l^{1-|b_{n-1}|} q\half m_2) = -l^{2-|b_{n-1}|}q\, m_2
\end{align*}
and therefore
\begin{gather*}
l^{|b_n|}\,m_2 = (-l^{|b_n|+|b_{n-1}|-1}q\inv)(-l^{2-|b_{n-1}|}q)\, m_2 = \mu(l,q)\, m_0.\end{gather*}

So, if $t_n = -1$ we have $P_0 = m_0\,\tilde{F}_0$. Now, suppose $t_n = 1$. Again, have four cases.

In type $(\ldots,-1,1)$, by Corollary~\ref{cor:FP-rec}, Remark~\ref{rmk:m}, and Definition~\ref{def:m},
\begin{align*}
\mu(l,q) &= 1 &
\nu(l,q) &= -l^2q &
m_0 &= -l^{-|b_n|-1}q\nhalf m_1 &
m_0 &= l^{-|b_n|}\, m_2
\end{align*}
and therefore $l^{-|b_n|}\,m_2 = \mu(l,q)\, m_0$ and
\begin{align*}
q\half \frac{1-l^{-|b_n|}}{l-l\inv}\,m_1 &= q\half l^{1-|b_n|} [|b_n|/2]_{l^2}\, m_1 \\
&= (-l^2q)[|b_n|/2]_{l^2}(-l^{-|b_n|-1}q\nhalf)\,m_1 \\
&= \nu(l,q)\, [|b_n|/2]_{l^2}\, m_0
\end{align*}
since $(1-l^{-|b_n|})/(l-l\inv) = l^{1-|b_n|}(1-l^{|b_n|})/(1-l^2) = l^{1-|b_n|}[|b_n|/2]_{l^2}$.

In type $(1,1)$, by Corollary~\ref{cor:FP-rec} and Definition~\ref{def:m},
\begin{align*}
\mu(l,q) &= -l^{|b_1|} w\inv &
\nu(l,q) &= 1 &
m_0 &= l^{1-|b_2|}q\half m_1
\end{align*}\begin{align*}
m_0 &= (l^{1-|b_2|}q\half)(-l^{-|b_1|-1}q\nhalf w m_2) = -l^{-|b_2|-|b_1|} w\, m_2
\end{align*}
and therefore $l^{-|b_2|}\,m_2 = (-l^{|b_1|}w\inv)(-l^{-|b_2|-|b_1|} w)\, m_2 = \mu(l,q)\, m_0$ and
\begin{align*}
q\half \frac{1-l^{-|b_2|}}{l-l\inv}\,m_1 = [|b_2|/2]_{l^2}(l^{1-|b_2|}q\half)\,m_1 = \nu(l,q)\, [|b_2|/2]_{l^2}\, m_0
\end{align*}
since as discussed above, $(1-l^{-|b_n|})/(l-l\inv) = l^{1-|b_n|}[|b_n|/2]_{l^2}$.

In type $(\ldots,-1,1,1)$, by Corollary~\ref{cor:FP-rec}, Definition~\ref{def:m}, and Remark~\ref{rmk:m},
\begin{align*}
\mu(l,q) &= -l^{|b_{n-1}|} &
\nu(l,q) &= 1 &
m_0 &= l^{1-|b_n|}q\half m_1
\end{align*}\begin{align*}
m_0 &= (l^{1-|b_n|}q\half)(-l^{-|b_{n-1}|-1}q\nhalf m_2) = -l^{-|b_n|-|b_{n-1}|}\, m_2
\end{align*}
and therefore $l^{-|b_n|}\,m_2 = (-l^{|b_{n-1}|})(-l^{-|b_n|-|b_{n-1}|})\, m_2 = \mu(l,q)\, m_0$ and
\begin{align*}
q\half \frac{1-l^{-|b_n|}}{l-l\inv}\,m_1 = [|b_n|/2]_{l^2}(l^{1-|b_n|}q\half)\,m_1 = \nu(l,q)\, [|b_n|/2]_{l^2}\, m_0
\end{align*}
since as discussed above, $(1-l^{-|b_n|})/(l-l\inv) = l^{1-|b_n|}[|b_n|/2]_{l^2}$.

In type $(\ldots,1,1,1)$, by Corollary~\ref{cor:FP-rec} and Definition~\ref{def:m},
\begin{align*}
\mu(l,q) &= l^{|b_{n-1}|-2}q\inv &
\nu(l,q) &= 1 &
m_0 &= l^{1-|b_n|}q\half m_1
\end{align*}\begin{align*}
m_0 &= (l^{1-|b_n|}q\half)(l^{1-|b_{n-1}|}q\half\, m_2) = l^{2-|b_n|-|b_{n-1}|}q\, m_2
\end{align*}
and therefore $l^{-|b_n|}\,m_2 = (l^{|b_{n-1}|-2}q\inv)(l^{2-|b_n|-|b_{n-1}|}q)\, m_2 = \mu(l,q)\, m_0$. The coefficients for $\tilde{F}_1$ are identical to the previous case.

Thus, if $n > 1$ we have $P_0 = m_0\,\tilde{F}_0$, which completes the proof.
\end{proof}

As an immediate result, we can apply the substitutions in Theorem~\ref{thm:homfly-to-alex}, to get a new specialization which yields the Alexander polynomial $\Delta(L)$. 
{\renewcommand{\thetheorem}{\ref{cor:C}}\begin{corollary}
For all even continued fractions $[b_1,\ldots,b_n]$ with $n \geq 0$,
\[ \Delta(C\big([b_1,\ldots,b_n]\big)) = \sgn(c_0)\,t^{e_0}\varphi_A(F[b_1,\ldots,b_n]) \]
where $c_0t^{e_0}$ is the leading term of $\Delta(C\big([b_1,\ldots,b_n]\big))$ and $\varphi_A(y_i) = -t^{(-1)^i}$.
\end{corollary}}
\begin{proof}
Under the substitutions $l = 1$, $q = t$ from Theorem~\ref{thm:homfly-to-alex}, the fraction $w$ becomes $-t$, so $\varphi_P(y_1)$, which is equal to either $l^2 w\inv$ or $q\inv w$, becomes $-t\inv$ in both cases. Furthermore, $\varphi_P(y_{2i})$, which is equal to $-l^2q$, becomes $-t$, and $\varphi_P(y_{2i+1})$, which is equal to $-q\inv$, becomes $-t\inv$. Thus, $\varphi_P(y_i)$ becomes $-t^{(-1)^i}$, which is exactly the definition of $\varphi_A$.

Applying these substitutions to Corollary~\ref{cor:homfly-unknot} and Lemma~\ref{lem:homfly-rec}, it is straightforward to verify, using similar reasoning to that in Section~5 of \cite{lee_cluster_2019}, that the degree of $m[b_1,\ldots,b_n]|_{l=1,q=t}$ is exactly the degree $e_0$ of $\Delta(C\big([b_1,\ldots,b_n]\big))$, and its coefficient the sign of $t^{e_0}$ in $\Delta(C\big([b_1,\ldots,b_n]\big))$. Thus, the desired equality follows from Theorem~\ref{thm:B}.
\end{proof}

\appendix

\section{Proof of Theorem~\ref{thm:path_poset_cf}}\label{appx:thm:path_poset_cf}

We begin by proving the result in the special case of positive continued fractions.

\begin{lemma}\label{lem:path_posets_pos_cf}
For any positive continued fraction $[a_1,\ldots,a_n]$ with $n \geq 0$, the poset $Q[a_1,\ldots,a_n]$ is isomorphic to $Q\big([a_1,\ldots,a_n]\big)$.
\end{lemma}
\begin{proof}
Since $Q[\,] = Q(\infty) = Q[1] = Q(1) = \emptyset$, we need only consider the case where $\ell_n \geq 2$. As discussed in the proof of Proposition~\ref{prop:path_posets_q}, 
it suffices to show that the signs of the edges of the Hasse diagram of $Q[a_1,\ldots,a_n]$ are exactly $\isgn[a_1,\ldots,a_n]$. Every edge of the straight segment $S_i$ has sign $t_i$ by definition, and there are exactly $a_i-2$ many of them. Furthermore, because $[a_1,\ldots,a_n]$ is positive, $t_i \neq t_{i+1}$ for all $i$ and therefore there is a connecting vertex between each $S_i$ and $S_{i+1}$. The sign of the edge on the left of each connecting vertex is $t_i$, and the sign on the right is $t_{i+1}$, and so the signs of the edges of $Q[a_1,\ldots,a_n]$ are
\[ (\underbrace{\underbrace{t_1,\ldots,t_1}_{a_1-2},t_1}_{a_1-1},\underbrace{t_2,\underbrace{t_2,\ldots,t_2}_{a_2-2},t_2}_{a_2},\ldots,\underbrace{t_n,\underbrace{t_n,\ldots,t_n}_{a_n-2}}_{a_n-1}), \]
which are exactly the terms of $\isgn[a_1,\ldots,a_n]$.
\end{proof}

The proof of Theorem~\ref{thm:path_poset_cf} is powered by the following lemma.

\begin{lemma}\label{lem:path_poset_cf}
For $k > 0$ and $n > k$, let $[a_1,\ldots,a_k,c_{k+1},\ldots,c_n] = p/q$ be any continued fraction which satisfies the conditions of Definition~\ref{def:path_poset_cf} such that its first $k$ terms $a_1,\ldots,a_k$ are strictly positive. Then, there exists some $k' > k$, a sign $s \in \{1,-1\}$, and strictly positive integers $a'_1,\ldots,a'_{k'}$ such that the continued fraction $[a'_1,\ldots,a'_{k'},sc_{k+2},\ldots,sc_n] = p/q$ and the posets $Q[a'_1,\ldots,a'_{k'},sc_{k+2},\ldots,sc_n]$ and $Q[a_1,\ldots,a_k,c_{k+1},\ldots,c_n]$ are isomorphic.
\end{lemma}
\begin{proof}
If $c_{k+1} > 0$, set $k' = k+1$, the sign $s = 1$, and
\[ a'_i = a_i \text{ for all } i \leq k, \quad\quad\quad\quad a_{k+1} = c_{k+1}. \]
The resulting continued fractions are identical, so the resulting path posets are also identical, and therefore certainly isomorphic.

Otherwise, $c_{k+1} < 0$. The following identity of continued fractions is straightforward to verify. \begin{equation}\label{eq:cfid1}\begin{aligned}
&[a_1,\ldots,a_k,c_{k+1},\ldots,c_n] \\&\quad
= [a_1,\ldots,a_{k-1},(a_k-1),1,(-c_{k+1}-1),-c_{k+2}\ldots,-c_n]
\end{aligned}\end{equation}
Thus, set $k' = k+2$, the sign $s = -1$, and
\[ a'_i = a_i \text{ for all } i < k, \quad\enspace a_{k} = (a_k - 1), \quad\enspace a_{k+1} = 1, \quad\enspace a_{k+2} = (-c_{k+1}-1). \]
The resulting continued fractions are equal by \eqref{eq:cfid1}, so all we need to show is that $Q_A = Q[a_1,\ldots,a_k,c_{k+1},\ldots,c_n]$ is isomorphic to
\[ Q_B = Q[a_1,\ldots,a_{k-1},(a_k-1),1,(-c_{k+1}-1),-c_{k+2}\ldots,-c_n]. \]

In the path poset $Q_A$ the sub-posets $S_k$ and $S_{k+1}$ are joined as shown on the left of Figure~\ref{fig:lem:path_poset_cf-1}, and in the path poset $Q_B$, the sub-posets $S_k$, $S_{k+1}$, and $S_{k+2}$ are joined as shown below on the right of Figure~\ref{fig:lem:path_poset_cf-1}. Notice that these sub-posets are isomorphic. For $i < k$ the segments $S_i$ and how they are joined in the two posets are identical, and for $i > k+1$, the segment $S_i$ in $Q_A$ is isomorphic to $S_{i+1}$ in $Q_B$ since $t_i$ for $[a_1,\ldots,a_k,c_{k+1},\ldots,c_n]$ is equal to $t_{i+1}$ for $[a_1,\ldots,a_{k-1},(a_k-1),1,(-c_{k+1}-1),-c_{k+2}\ldots,-c_n]$, and the same for $|c_i|$ and $|c_{i+1}|$. So, the two posets are isomorphic, as desired.
\end{proof}

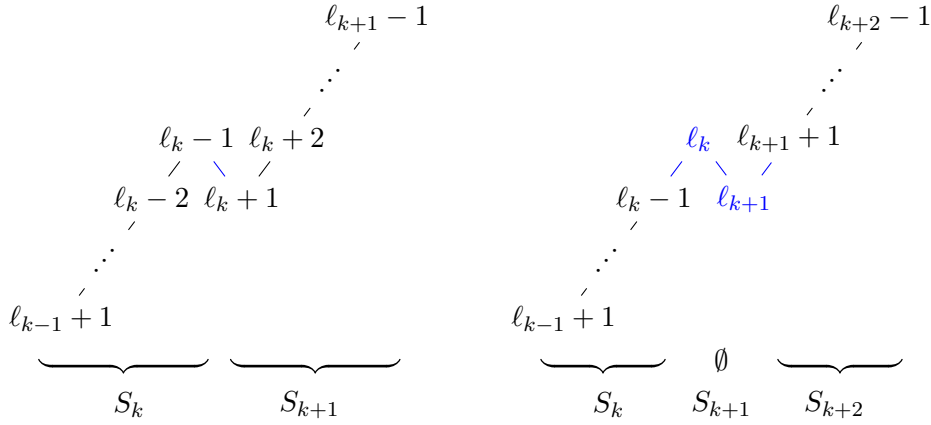
\begin{figure}[ht]
$\begin{gathered}\begin{tikzpicture}[scale=0.8,xscale=0.75]
\node (lip1) at (1,1) {$\ell_{k-1}+1$};
\node (di) at (2,2) {\rotatebox{51.34}{$\cdots$}};
\node (lsim2) at (3,3) {$\ell_k - 2$};
\node (lsim1) at (4,4) {$\ell_k - 1$};
\node (lsip1) at (5,3) {$\ell_k+1$};
\node (lsip2) at (6,4) {$\ell_k+2$};
\node (dsi) at (7,5) {\rotatebox{51.34}{$\cdots$}};
\node (lsssim1) at (8,6) {$\ell_{k+1} - 1$};
\draw (lip1) -- (di) -- (lsim2) -- (lsim1);
\draw[color=blue] (lsim1) -- (lsip1);
\draw (lsip1) -- (lsip2) -- (dsi) -- (lsssim1);
\draw[decorate, decoration={calligraphic brace,amplitude=5pt}, line width=1pt]
    ( $ (4+0.25,0.33) $ ) -- ( $ (1-0.5,0.33) $ );
\node at (2.5,0.33-0.75) {$S_k$};
\draw[decorate, decoration={calligraphic brace,amplitude=5pt}, line width=1pt]
    ( $ (8+0.5,0.33) $ ) -- ( $ (5-0.25,0.33) $ );
\node at (6.5,0.33-0.75) {$S_{k+1}$};
\end{tikzpicture}\end{gathered} 
\quad\quad
\begin{gathered}\begin{tikzpicture}[scale=0.8,xscale=0.75]
\node (lip1) at (1,1) {$\ell_{k-1}+1$};
\node (di) at (2,2) {\rotatebox{51.34}{$\cdots$}};
\node (lsim2) at (3,3) {$\ell_k - 1$};
\node (lsim1) at (4,4) {$\textcolor{blue}{\ell_k}$};
\node (lsip1) at (5,3) {$\textcolor{blue}{\ell_{k+1}}$};
\node (lsip2) at (6,4) {$\ell_{k+1}+1$};
\node (dsi) at (7,5) {\rotatebox{51.34}{$\cdots$}};
\node (lsssim1) at (8,6) {$\ell_{k+2} - 1$};
\draw (lip1) -- (di) -- (lsim2);
\draw[color=blue] (lsim2) -- (lsim1) -- (lsip1) -- (lsip2);
\draw (lsip2) -- (dsi) -- (lsssim1);
\draw[decorate, decoration={calligraphic brace,amplitude=5pt}, line width=1pt]
    ( $ (3+0.25,0.33) $ ) -- ( $ (1-0.5,0.33) $ );
\node at (2,0.33-0.75) {$S_k$};
\node at (4.5,0.33) {$\emptyset$};
\node at (4.5,0.33-0.75) {$S_{k+1}$};
\draw[decorate, decoration={calligraphic brace,amplitude=5pt}, line width=1pt]
    ( $ (8+0.5,0.33) $ ) -- ( $ (6-0.25,0.33) $ );
\node at (7,0.33-0.75) {$S_{k+2}$};
\end{tikzpicture}\end{gathered}$
\caption{Sub-posets of $Q_A$ (on the left) and $Q_B$ (on the right).}
\label{fig:lem:path_poset_cf-1}
\end{figure}

\begin{proof}[Proof of Theorem~\ref{thm:path_poset_cf}]
Let $[c_1,\ldots,c_n] = p/q$ be any continued fraction which satisfies the conditions of Definition~\ref{def:path_poset_cf}. It suffices to show that the poset $Q[c_1,\ldots,c_n]$ is isomorphic to either $Q(p/q)$ or $Q(p/(q-\sgn(q)\,p))$. If $n = 0$, we have that $Q[\,] = Q(\infty) = \emptyset$, so assume $n \geq 1$.

First suppose $c_1 > 0$. In that case, we can apply Lemma~\ref{lem:path_poset_cf} $(n-1)$ many times to obtain a positive continued fraction $[a_1,\ldots,a_{n'}] = p/q$ such that the posets $Q[c_1,\ldots,c_n]$ and $Q[a_1,\ldots,a_{n'}]$ are isomorphic. Thus, by Lemma~\ref{lem:path_posets_pos_cf}, we conclude that $Q[c_1,\ldots,c_n]$ is isomorphic to $Q(p/q)$.

Otherwise, $c_1 < 0$ and therefore $\sgn(q) = -1$. We will show that $Q[c_1,\ldots,c_n]$ is isomorphic to $Q(p/(q+p))$. The following identity is straightforward to verify.
\begin{equation}
[1,(-c_1-1),-c_2,\ldots,-c_n] = 1 + \frac{1}{-\frac{p}{q}-1} = \frac{p}{q+p}
\end{equation}
Thus, consider the poset $Q[1,(-c_1-1),-c_2,\ldots,-c_n]$. In this poset, the sub-posets $S_1$ and $S_2$ are joined as shown on the right of Figure~\ref{fig:thm:path_poset_cf-1}. Shown on the left of Figure~\ref{fig:thm:path_poset_cf-1} is the segment $S_1$ in $Q[c_1,\ldots,c_n]$.

\begin{figure}
$\begin{gathered}\begin{tikzpicture}[scale=0.8,xscale=0.75]
\node (lsip1) at (5,-3) {$1$};
\node (lsip2) at (6,-4) {$2$};
\node (dsi) at (7,-5) {\rotatebox{-51.34}{$\cdots$}};
\node (lsssim1) at (8,-6) {$\ell_1 - 1$};
\draw (lsip1) -- (lsip2) -- (dsi) -- (lsssim1);
\draw[decorate, decoration={calligraphic brace,amplitude=5pt}, line width=1pt]
    ( $ (8+0.5,-6.67) $ ) -- ( $ (5-0.25,-6.67) $ );
\node at (6.5,-6.67-0.75) {$S_1$};
\end{tikzpicture}\end{gathered}
\quad\quad
\quad\quad
\begin{gathered}\begin{tikzpicture}[scale=0.8,xscale=0.75]
\node (lsip1) at (5,-3) {$\textcolor{blue}{1}$};
\node (lsip2) at (6,-4) {$2$};
\node (dsi) at (7,-5) {\rotatebox{-51.34}{$\cdots$}};
\node (lsssim1) at (8,-6) {$\ell_2 - 1$};
\draw[color=blue] (lsip1) -- (lsip2);
\draw (lsip2) -- (dsi) -- (lsssim1);
\node at (4,-4) {$\emptyset$};
\node at (4,-4-0.75) {$S_1$};
\draw[decorate, decoration={calligraphic brace,amplitude=5pt}, line width=1pt]
    ( $ (8+0.5,-6.67) $ ) -- ( $ (6-0.25,-6.67) $ );
\node at (7,-6.67-0.75) {$S_2$};
\end{tikzpicture}\end{gathered}$
\caption{Sub-posets of $Q[c_1,\ldots,c_n]$ (on the left) and $Q[1,(-c_1-1),-c_2,\ldots,-c_n]$ (on the right).}
\label{fig:thm:path_poset_cf-1}
\end{figure}
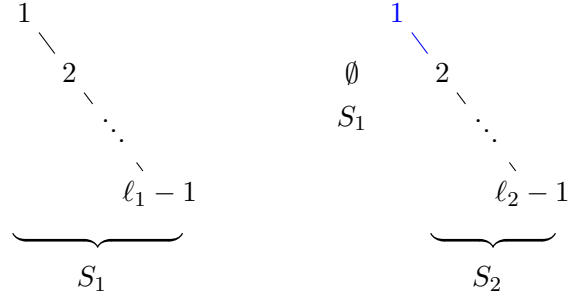

For $i > 1$, the segment $S_i$ in the poset $Q[c_1,\ldots,c_n]$ is identical to $S_{i+1}$ in the poset $Q[1,(-c_1-1),-c_2,\ldots,-c_n]$ since $t_i$ for the continued fraction $[c_1,\ldots,c_n]$ is equal to $t_{i+1}$ for the continued fraction $[1,(-c_1-1),-c_2,\ldots,-c_n]$, and the same for $|c_i|$ and $|c_{i+1}|$. Thus, $Q[c_1,\ldots,c_n]$ and $Q[1,(-c_1-1),-c_2,\ldots,-c_n]$ are identical, and therefore certainly isomorphic. By the argument in the previous case, we know that $Q[1,(-c_1-1),-c_2,\ldots,-c_n]$ is isomorphic to $Q(p/(q+p))$. Thus, $Q[c_1,\ldots,c_n]$ is isomorphic to $Q(p/(q+p))$, as desired.
\end{proof}

\printbibliography

\end{document}